\documentclass[12pt]{article}

\usepackage{amsmath}
\usepackage{amssymb}
\usepackage{latexsym}

\newtheorem{assumption}{Assumption}

\def\qed{ \ \vrule width.2cm height.2cm depth0cm\smallskip}

\newenvironment{proof}{\noindent {\bf Proof.\/}}{$\qed$\vskip 0.1in}

\newcommand{\la}{\langle}
\newcommand{\ra}{\rangle}

\newcommand{\ba}{\begin{array}}
\newcommand{\ea}{\end{array}}
\newcommand{\be}{\begin{equation}}
\newcommand{\ee}{\end{equation}}
\newcommand{\bea}{\begin{eqnarray}}
\newcommand{\eea}{\end{eqnarray}}
\newcommand{\beaa}{\begin{eqnarray*}}
\newcommand{\eeaa}{\end{eqnarray*}}

\def\dbD{\mathbb{D}}
\def\dbE{\mathbb{E}}
\def\dbF{\mathbb{F}}

\def\dbP{\mathbb{P}}
\def\dbR{\mathbb{R}}

\def\a{\alpha}

\def\d{\delta}
\def\e{\varepsilon}

\def\l{\lambda}

\def\si{\sigma}
\def\t{\tau}

\def\o{\omega}

\def\D{\Delta}

\def\O{\Omega}
\def\cF{{\cal F}}

\def\cN{{\cal N}}

\def\cP{{\cal P}}

\def\no{\noindent}

\def\ms{\medskip}
\def\bs{\bigskip}
\def\q{\quad}
\def\qq{\qquad}

\def\cd{\cdot}
\def\cds{\cdots}

\def\qed{ \hfill \vrule width.25cm height.25cm depth0cm\smallskip}

\newcommand{\basa}{\begin{assumption}}
\newcommand{\easa}{\end{assumption}}

\newcommand{\bas}{\begin{assum}}
\newcommand{\eas}{\end{assum}}

\def\esup{\mathop{\rm ess\;sup}}

 \def\cd{\cdot}
\def\cds{\cdots}

\def\sgn{\hbox{\rm sgn$\,$}}

\def\dis{\displaystyle}

\def\cad{{c\`{a}dl\`{a}g}}

\def\1{\mathbf{1}}

\def\:{\!:\!}
\def\reff#1{{\rm(\ref{#1})}}

at 9pt

\begin{document}

\newtheorem{thm}{Theorem}[section]
\newtheorem{lem}[thm]{Lemma}
\newtheorem{cor}[thm]{Corollary}
\newtheorem{prop}[thm]{Proposition}
\newtheorem{rem}[thm]{Remark}
\newtheorem{eg}[thm]{Example}
\newtheorem{defn}[thm]{Definition}
\newtheorem{assum}[thm]{Assumption}

\renewcommand {\theequation}{\arabic{section}.\arabic{equation}}
\def\thesection{\arabic{section}}

\title{\bf Some Norm Estimates for  Semimartingales}
\author{Triet {\sc Pham}\footnote{University of Southern California, Department of Mathematics, trietpha@usc.edu. Research supported by USC Graduate School Dissertation Completion Fellowship.}
        \and Jianfeng {\sc Zhang}\footnote{University of Southern California, Department of Mathematics, jianfenz@usc.edu. Research supported in part by NSF grant DMS 1008873.}
}\maketitle

\begin{abstract} In this paper we introduce a new type of norms for semimartingales, under both linear and nonlinear expectations.  Our norm is defined in the spirit of  quasimartingales, and it characterizes square integrable semimartingales. This work is motivated by our study of zero-sum stochastic differential games \cite{PZ}, whose value process is conjectured to be a semimartingale under a class of probability measures.  As a by product, we establish some a priori estimates for doubly reflected BSDEs without imposing  the  Mokobodski's condition directly. 

\vspace{5mm}


\end{abstract}

\section{Introduction}
\setcounter{equation}{0}

For an optimization problem under volatility uncertainty, the value process can be characterized as the unique solution of a second order backward SDE, introduced by Cheridito, Soner, Touzi, and Victoir \cite{CSTV} and Soner, Touzi, and Zhang \cite{STZ-2BSDE}, and as the unique viscosity solution of a path dependent HJB equation, introduced by Ekren, Keller, Touzi, and Zhang \cite{EKTZ} and Ekren, Touzi, and Zhang \cite{ETZ0, ETZ1, ETZ2}. See also the closely related concepts $G$-martingale introduced by Peng \cite{Peng-G} and $G$-BSDE studied by Hu, Ji, Peng, and Song \cite{HJPS}.  This value process is a supermartingale (or, in general case, a $g$-supermartingale as introduced in Peng \cite{Peng-g}), under the associated non-dominated class of mutually singular probaility measures. In Pham and Zhang \cite{PZ}, we studied a zero sum stochastic differential game and characterized the game value process as the unique viscosity solution of a path dependent Bellman-Isaacs equation.  It is natural to conjecture that, under certain technical conditions, this value process should be a semimartingale under the underlying class of probability measures, which will enable us to characterize the value process as the solution to an extended second order BSDE with a non-convex generator. This requires a systematic study on square integrable semimartingales, namely semimartingales whose martingale part and total variation of its  finite variation part are square integrable, under both linear  and nonlinear expectations.

Our first goal of this paper is to introduce a norm which characterizes square integrable semimartingales, under a fixed (linear) probability measure. Our norm is strongly motived from the definition of quasimartingales. The main feature is that the norm involves only the semimartingale itself, without involving directly  its decomposition. This is important in applications because the semimartingale under consideration is typically a value process, , e.g. in \cite{PZ},  and thus has a representation.  We prove that a progressively measurable process is a square integrable semimartingale if and only if it has finite norm in our sense.

We next extend our norm to semimartingales under nonlinear expectations, typically the $G$-expectation of Peng \cite{Peng-G}. As observed in Soner, Touzi and Zhang \cite{STZ-G}, a $G$-martingale is a supermartingale under a class of probability measures. It is clear that a  $G$-supermartingale satisfies the same property. However, a $G$-submartingale is in general neither a supermartingale nor a submartingale under each probability measure. We show that, any progressively measurable process with finite norm under $G$-expectation in our sense  has to be a semimartingale under each probability measure.

Our long term goal is to apply our norm, or its variations if necessary, to study the structure of general $G$-semimartingales and to study the semimartingale property of the viscosity solution of path dependent Bellman-Isaacs equations. The latter can also be viewed as regularity of viscosity solutions of path dependent PDEs. We hope to address these issues in future research.  

As a by product, we also provide a tractable sufficient condition for the wellposedness of doubly reflected backward SDEs (DRBSDE, for short) without imposing directly the Mokobodski's condition. There are typically two approaches in the literature. One is to assume the Mokobodski's condition, namely there exists a square integrable semimartingale between the two given barriers, see e.g. Cvitanic and Karatzas \cite{CK} and Peng and Xu \cite{PX}, and the other is to use local solutions, see e.g. Hamadene and Hassani \cite{HH} and Hamadene, Hassani and Ouknine  \cite{HHO}.   The latter approach, while easy to verify its conditions, does not yield any norm estimates. We remark that such estimates are important in applications, for example when one considers discritization of DRBSDEs, see e.g.   Chassagneux \cite{Chassagneux}. Inspired by our norm for square integrable semimartingales, we introduce a norm for the barriers of the DRBSDEs, which is equivalent to the  Mokobodski's condition but is easier to verify. Moreover, we provide the sensitivity analysis of DRBSDEs with respect to its barriers. Such an estimate seems new in the literature and is important in numerical discretization of DRBSDEs.

The rest of the paper will be organized as follows. In next section we introduce the norm for semimartingales under a fixed probability measure and obtain the estimates. In Section 3 study DRBSDEs by introducing a norm in the same spirit. In Section 4 we extend the norm to the $G$-framework. Finally in Appendix we provide some additional results.

\section{Norm Estimates for Semimartingales}
\setcounter{equation}{0}

Let $T>0$ be fixed, $(\O, \cF, \dbF, \dbP)$ be a filtered probability space on  $[0,T]$, and $\dbD(\dbF)$ be the space of $\dbF$-progressively measurable {\cad} processes.  We shall always assume
 \bea
\label{Filtration}
\left.\ba{c}
\mbox{ $\dbF$ is right continuous and its $\dbP$-augmentation}~~\bar\dbF^\dbP~~\mbox{is a Brownian filtration},\\ 
\mbox{and consequently, any $\dbF$-martingale $M$ is continuous, $\dbP$-a.s.}
\ea\right.
\eea
We note that the filtration $\dbF$ is not necessarily complete under $\dbP$. The removal of the completeness requirement will be important in Section \ref{sect-G} below. However, the following simple lemma, see e.g. \cite{STZ-aggregation}, shows that we may assume all the processes involved in this section are $\dbF$-progressively measurable.

\begin{lem}
\label{lem-filtration}
 For any $\bar\dbF^\dbP$-progressively measurable process $X$, there exists a unique  ($dt\times d\dbP$-a.s.) $\dbF$-progressively measurable process $\tilde X$ such that $\tilde X = X$, $dt\times d\dbP$-a.s. Moreover, if $X$ is \cad, $\dbP$-a.s., then so is $\tilde X$.
\end{lem}

We recall that a semimartingale $Y\in \dbD(\dbF)$ has the following decomposition: 
\bea
\label{semimg}
Y_t = Y_0 + M_t + A_t,
\eea
where $M$ is a martingale, $A$ has finite variation, and $M_0=A_0=0$.  
Now given a  process $Y\in \dbD(\dbF)$,
we are interested in the following  questions:

(i) Is $Y$ a semimartingale?

(ii) Do we have appropriate norm estimates for $Y$, $M$, and $A$?

\ms

\no The first question was answered by Bichteler-Dellacherie, see e.g. \cite{Protter} and Appendix of this paper for some further discussion. The main goal of this section is to answer the second question. As explained in Introduction, the latter question is natural and important for our study of semimartingales under nonlinear expectations.

\subsection{Some preliminary results}

We first note that, when $Y$ is a supermartingale or submartingale, it is well known that $Y$ is a semimartingale and the following norm estimates hold. Since the arguments will be important for our general case, we provide the proof for completeness.

\begin{lem}
\label{lem-super}
Let \reff{Filtration} hold. There exist universal constants $0<c<C$ such that, for any $Y$ in the form of \reff{semimg} with monotone $A$, it holds 
\bea
\label{super-est}
c\|Y\|_{\dbP,0}^2\le \dbE^\dbP\Big[|Y_0|^2 + \la M\ra_T + |A_T|^2\Big] \le C\|Y\|_{\dbP,0}^2.
\eea
where, for any $Y\in \dbD(\dbF)$,
\bea
\label{Y0}
\|Y\|_{\dbP,0}^2 := \dbE^\dbP\Big[ \sup_{0\le t\le T}|Y_t|^2\Big].
\eea
\end{lem}

\begin{proof}
The first inequality is obvious. We shall only prove the second inequality. By otherwise using the standard stopping techniques, we may assume without loss of generality that
$
 \dbE^\dbP[\sup_{0\le t\le T}|Y_t|^2+ \la M\ra_T + |A_T|^2]<\infty.
 $

 Apply It\^{o}'s formula and recall \reff{Filtration} that $M$ is continuous, we have
 \bea
 \label{super-est1}
 Y_T^2 = Y_0^2 + \la M\ra_T + 2\int_0^T Y_t dM_t + 2\int_0^T Y_{t-} dA_t + \sum_{0< t\le T} |\D Y_t|^2.
 \eea
 Note that
 \beaa
 \dbE^\dbP\Big[\Big(\int_0^T |Y_t|^2 d\la M\ra_t\Big)^{1\over 2}\Big]\le  \dbE^\dbP\Big[\sup_{0\le t\le T} |Y_t|  (\la M\ra_T)^{1\over 2}\Big]\le {1\over 2}  \dbE^\dbP\Big[\sup_{0\le t\le T} |Y_t|^2 +  \la M\ra_T\Big]<\infty.
 \eeaa
 Then $Y_t dM_t$ is a true martngale, and  thus, for any $\e>0$, it follows from  \reff{super-est1} and the monotonicity of $A$ that
  \bea
  \label{super-est2}
 \dbE^\dbP[\la M\ra_T] &\le& \dbE^\dbP\Big[\la M\ra_T+ \sum_{0\le t\le T} |\D Y_t|^2\Big] = \dbE^\dbP\Big[ Y_T^2 -Y_0^2 -  2\int_0^T Y_{t-} dA_t\Big]\\
  &\le& \dbE^\dbP\Big[ |Y_T|^2  +|Y_0|^2+  2\sup_{0\le t\le T}|Y_t| |A_T|\Big]
   \le C\e^{-1} \|Y\|_{\dbP,0}^2 + \e \dbE^\dbP[|A_T|^2].\nonumber
  \eea
  Moreover, note that
  $
  A_T = Y_T-Y_0 - M_T.
  $
Then \reff{super-est2} leads to
  \beaa
  \dbE^\dbP[|A_T|^2] \le C\|Y\|_{\dbP,0}^2 + C\dbE^\dbP[\la M\ra_T] \le C\e^{-1} \|Y\|_{\dbP,0}^2 + C\e \dbE^\dbP[|A_T|^2].
  \eeaa
  Set $\e := {1\over 2C}$ for the above $C$, we obtain
  $
   \dbE^\dbP[|A_T|^2] \le C\|Y\|_{\dbP,0}^2.
 $
 This, together with \reff{super-est2}, proves the second inequality.
 \end{proof}
 
 The next lemma is a discrete version of Lemma \ref{lem-super}. Since the arguments are very similar, we omit the proof. 
 \begin{lem}
 \label{lem-super2} 
Let $0=\t_0\le\cds\le\t_n=T$ be a sequence of stopping times. In the setting of Lemma \ref{lem-super},  if $A_{\t_i}\in \cF_{\t_{i-1}}$, then
\bea
\label{super-est3}
c\dbE^\dbP\Big[\max_{0\le i\le n}|Y_{\t_i}|^2\Big] \le  \dbE^\dbP\Big[|Y_0|^2 + \la M\ra_T + |A_T|^2\Big] \le C\dbE^\dbP\Big[\max_{0\le i\le n}|Y_{\t_i}|^2\Big].
\eea 
\end{lem}

\subsection{Square integrable semimartingales}
 In this subsection we characterize the norm for square integrable semimartingales. For $0\le t_1<t_2\le T$, let $\dis\bigvee_{t_1}^{t_2} A$ denote the total variation of $A$ over the interval $(t_1, t_2]$.

\begin{defn}
We say a semimartingale $Y$ in the form of \reff{semimg} is a square integrable semimartingale if
\bea
\label{Y2}
\dbE^\dbP\Big[ |Y_0|^2 + \la M\ra_T + \Big(\bigvee_0^T A\Big)^2\Big]<\infty.
\eea
\end{defn}
We remark that \reff{Y2} is the norm used in standard literature for semimartingales, see e.g. \cite{Protter}. Clearly, for a square integrable semimartingale $Y$, we have $\|Y\|_{\dbP,0}<\infty$. However,  when $A$ is not monotone, in general the left side of \reff{Y2} cannot be dominated by $C\|Y\|_{\dbP,0}^2$. See Example \ref{eg-Y0} below.

 Our goal is to characterize square integrable semimartingales through the process $Y$ itself, without involving $M$ and $A$ directly. In many applications, see e.g. \cite{PZ}, we may have a representation formula for the process $Y$, but in general it is difficult to obtain representation formulae for $M$ and $A$. So  conditions imposed on $Y$ are more tractable than those on $M$ and $A$. We introduce the following norm:
 \bea
 \label{YP}
\|Y\|_\dbP^2 := \|Y\|_{\dbP,0}^2 + \sup_\pi \dbE^\dbP\Big[\Big(\sum_{i=0}^{n-1}\big|\dbE^\dbP_{\t_i}(Y_{\t_{i+1}})-Y_{\t_i}\big|\Big)^2\Big], ~\mbox{for any}~ Y\in \dbD(\dbF),
 \eea
 where the supremum is over all partitions $\pi: 0=\t_0\le \cds\le \t_n=T$ for some stopping times $\t_0,\cds,\t_n$.

 \begin{rem}
\label{rem-quasimg}
{\rm Our norm $\|\cd\|_\dbP$ is strongly motivated from the definition of  quasimartingale:  a  process $Y\in \dbD(\dbF)$ is a called a quasimartingale if
 \bea
 \label{quasimg}
  \sup_\pi \dbE^\dbP\Big[\sum_{i=0}^{n-1}\big|\dbE^\dbP_{t_i}(Y_{t_{i+1}})-Y_{t_i}\big|\Big]<\infty,
 \eea
where $\pi$ is a deterministic partition of $[0,T]$. We refer to  Rao \cite{Rao} and Dellacherie and Meyer \cite{DM} for the theory of quasimartingales, see also Meyer and Zheng \cite{MZ}. It is clear that our definition imposes stronger condition on $Y$: $\|Y\|_\dbP < \infty$ implies $Y$ is a quasimartingale.}
\qed
\end{rem}


The following a priori estimate is the main technical result of the paper.
 \begin{thm}
 \label{thm-semimg-est}
There exist universal constants $0<c<C$ such that, for any square integrable semimartingale $Y_t = Y_0 +M_t + A_t$,
\bea
\label{semi-est}
c\|Y\|_{\dbP}^2\le \dbE^\dbP\Big[ |Y_0|^2 + \la M\ra_T + \Big(\bigvee_0^T A\Big)^2\Big] \le C\|Y\|_{\dbP}^2.
\eea
 \end{thm}
	
 \begin{proof} (i) We first prove the left inequality. Let $\pi: 0=\t_0\le \cds\le \t_n=T$  be an arbitrary partition, and denote $\D A_{\t_{i+1}}:= A_{\t_{i+1}}-A_{\t_i}$. Then
 \bea
 \label{semi-est1}
&&\dbE^\dbP\Big[\Big(\sum_{i=0}^{n-1}\big|\dbE^\dbP_{\t_i}(Y_{\t_{i+1}})-Y_{\t_i}\big|\Big)^2\Big]= \dbE^\dbP\Big[\Big(\sum_{i=0}^{n-1}\big|\dbE^\dbP_{\t_i}(A_{\t_{i+1}})-A_{\t_i}\big|\Big)^2\Big]\nonumber\\
 &\le& \dbE^\dbP\Big[\Big(\sum_{i=0}^{n-1}\dbE^\dbP_{\t_i}(|\D A_{\t_{i+1}}|)\Big)^2\Big] =  \dbE^\dbP\Big[\Big(\sum_{i=0}^{n-1}[\dbE^\dbP_{\t_i}(|\D A_{\t_{i+1}}|) - |\D A_{\t_{i+1}}|] + \sum_{i=0}^{n-1} |\D A_{\t_{i+1}}| \Big)^2\Big]\nonumber \\
 &\le& C \dbE^\dbP\Big[\Big(\sum_{i=0}^{n-1}[\dbE^\dbP_{\t_i}(|\D A_{\t_{i+1}}|) - |\D A_{\t_{i+1}}|]\Big)^2\Big]  + C\dbE^\dbP\Big[ \Big(\bigvee_0^T A\Big)^2\Big] .
 \eea
 Note that
 \beaa
 \sum_{i=0}^{j}[\dbE^\dbP_{\t_i}(|\D A_{\t_{i+1}}|) - |\D A_{\t_{i+1}}|], j=0,\cds, n-1, ~~\mbox{is a martingale}.
 \eeaa
 Then
 \beaa
&& \dbE^\dbP\Big[\Big(\sum_{i=0}^{n-1}[\dbE^\dbP_{\t_i}(|\D A_{\t_{i+1}}|) - |\D A_{\t_{i+1}}|]\Big)^2\Big]  = \dbE^\dbP\Big[\sum_{i=0}^{n-1}\big[\dbE^\dbP_{\t_i}(|\D A_{\t_{i+1}}|) - |\D A_{\t_{i+1}}|\big]^2\Big]\\
&\le&  C \dbE^\dbP\Big[\sum_{i=0}^{n-1}\big[\big(\dbE^\dbP_{\t_i}(|\D A_{\t_{i+1}}|)\big)^2 + |\D A_{\t_{i+1}}|^2\big]\Big] \le  C \dbE^\dbP\Big[\sum_{i=0}^{n-1}\big[\dbE^\dbP_{\t_i}(|\D A_{\t_{i+1}}|^2) + |\D A_{\t_{i+1}}|^2\big]\Big]\\
&\le&  C \dbE^\dbP\Big[\sum_{i=0}^{n-1} |\D A_{\t_{i+1}}|^2\Big]\le C \dbE^\dbP\Big[\big(\sum_{i=0}^{n-1} |\D A_{\t_{i+1}}|\big)^2\Big] \le C\dbE^\dbP\Big[ \Big(\bigvee_0^T A\Big)^2\Big] .
 \eeaa
 This, together with \reff{semi-est1} and the left inequality of \reff{super-est}, proves the left inequality of \reff{semi-est}.

(ii)  We now prove the right inequality. First, for any $\e>0$, following the arguments in Lemma \ref{lem-super} one can easily  show that
 \bea
 \label{semi-est2}
\dbE^\dbP [\la M\ra_T  ]  \leq C\e^{-1} \|Y\|_{\dbP,0}^2 + \e \dbE^\dbP\Big[ \big(\bigvee_{0}^T A\big)^2 \Big].
\eea
We claim that
\bea
\label{semi-claim}
\dbE^\dbP\Big[ \big(\bigvee_{0}^T A\big)^2 \Big] \le C\|Y\|_{\dbP}^2 + C\dbE^\dbP [\la M\ra_T  ] .
\eea
This, together with \reff{semi-est2} and by choosing $\e$ small enough, implies the right inequality of \reff{semi-est} immediately.

We prove \reff{semi-claim}  in four steps.

\ms
 {\it Step1.}  Let $\pi:  0 = \t_0 \le \t_1 \le... \le \t_n = T$ be an arbitrary partition. Note that
\beaa
\dbE^\dbP_{\t_{i}}[Y_{\t_{i+1}}] - Y_{\t_{i}} =  \dbE^\dbP_{\t_{i}}[A_{\t_{i+1}}] - A_{\t_{i}}.
 \eeaa
Then
\beaa
\sum_{i = 0}^{n-1} \big[ A_{\t_{i+1}} - \dbE^\dbP_{\t_{i}}[A_{\t_{i+1}}] \big] &=& A_T - \sum_{i=0}^{n-1}\big(\dbE^\dbP_{\t_{i}}[A_{\t_{i+1}}]- A_{\t_{i}} \big)\\
&=& Y_T - Y_0 - M_T - \sum_{i=0}^{n-1}\big(\dbE^\dbP_{\t_{i}}[Y_{\t_{i+1}}] - Y_{\t_{i}}\big).
\eeaa
By the definition of $\|Y\|_\dbP$ \reff{YP}, we see that
\beaa
 \dbE^\dbP \Big[ \Big(\sum_{i = 0}^{n-1} \big[ A_{\t_{i+1}} - \dbE^\dbP_{\t_{i}}[A_{\t_{i+1}}] \big]\Big)^2\Big]\le C\|Y\|_\dbP^2 + C\dbE^\dbP[\la M\ra_T  ] .
 \eeaa
Note that
\beaa
\sum_{i=0}^{j-1}\big[ A_{\t_{i+1}} - \dbE^\dbP_{\t_{i}}[A_{\t_{i+1}}] \big],~~ j=1,\cds, n,~~\mbox{is a martingale}.
\eeaa
Then
\bea
\label{semi-est3}
 \dbE^\dbP \Big[ \sum_{i = 0}^{n-1} \big[ A_{\t_{i+1}} - \dbE^\dbP_{\t_{i}}[A_{\t_{i+1}}] \big]^2\Big]\le C\|Y\|_\dbP^2 + C\dbE^\dbP[\la M\ra_T  ] .
 \eea

 {\it Step 2.} In this step we assume   $A_t = \int_0^t a_s dK_s $, where $K$ is a continuous nondecreasing process and $a$ is a simple process. That is,
  \beaa
 a = a_{t_0} \mathbf{1}_{\{t_0\}} + \sum_{i=0}^{n-1} a_{t_i} \mathbf{1}_{(t_i,t_{i+1}]}&\mbox{for some}& 0=t_0<\cds<t_n=T, \; a_{t_i} \in \cF_{t_i}.
 \eeaa
Then, denoting $\a_i :=\sgn(a_{t_i})\in \cF_{t_i}$,
\beaa
\bigvee_0^T A &=& \int_0^T |a_t| dK_t =  \sum_{i=0}^{n-1}  \int_{t_i}^{t_{i+1}} \a_{i} a_t dK_t = \sum_{i=0}^{n-1} \a_{i} [A_{t_{i+1}} - A_{t_{i}}]\\
&=& \sum_{i=0}^{n-1} \a_{i}\big(A_{t_{i+1}} - \dbE^\dbP_{t_{i}}[A_{t_{i+1}}]\big) +  \sum_{i=0}^{n-1} \a_{i}\big (\dbE^\dbP_{t_{i}}[A_{t_{i+1}} ]- A_{t_{i}} \big) .
\eeaa
Note that
\beaa
\sum_{i=0}^{j} \a_{i}\big(A_{t_{i+1}} - \dbE^\dbP_{t_{i}}[A_{t_{i+1}}]\big), ~~j=0,\cds, n-1,~~\mbox{is a martingale}.
\eeaa
Then
\beaa
\dbE^\dbP\Big[\big(\bigvee_0^T A  \big)^2\Big]& \le& C\dbE^\dbP\Big[\sum_{i=0}^{n-1} \big|A_{t_{i+1}} - \dbE^\dbP_{t_{i}}[A_{t_{i+1}}]\big|^2+\Big(\sum_{i=0}^{n-1} \big|\dbE^\dbP_{t_{i}}[A_{t_{i+1}} ]- A_{t_{i}} \big| \Big)^2\Big].
\eeaa
By \reff{semi-est3} and the definition of $\|Y\|_\dbP$ \reff{YP} we obtain  \reff{semi-claim}.

\ms
{\it Step 3.} We now prove \reff{semi-claim} for general continuous process $A$ .  Denote $\dis K_t := \bigvee_0^t A$. Since $A$ is continuous, $K$ is also continuous. Moreover $d A_t$ is absolutely continuous with respect to $d K_t$ and thus $d A_t = a_t dK_t$ for some $a$. By \cite{KS}, Chapter 3 Lemma 2.7, for every $\e > 0$ there exists a simple process $\{a^\e\}$ such that

\bea
\label{semi-ae}
\dbE^\dbP\Big[\Big(\int_0^T |a^\e_t - a_t|dK_t\Big)^2\Big] \le \e.
\eea
Denote
\beaa
A^\e_t := \int_0^t a^\e_s dK_s,\q Y^\e_t := Y_0 + M_t + A^\e_t.
\eeaa
Then by {\it Step 2} we see that
\bea
\label{semi-est5}
\dbE^\dbP\Big[ \big(\bigvee_{0}^T A^\e \big)^2 \Big] \le C\|Y^\e\|_{\dbP}^2 + C\dbE^\dbP[\la M\ra_T  ]  .
\eea

Note that
\beaa
\bigvee_{0}^T A \le \bigvee_{0}^T A^\e + \bigvee_{0}^T [A^\e-A] \le \bigvee_{0}^T A^\e + \int_0^T |a^\e_t - a_t|dK_t.
\eeaa
Then
\bea
\label{semi-est6}
\dbE^\dbP\Big[ \big(\bigvee_{0}^T A \big)^2 \Big] \le C\dbE^\dbP\Big[ \big(\bigvee_{0}^T A^\e \big)^2 \Big]  + C\e.
\eea
On the other hand, apply the left inequality of \reff{semi-est} on $Y^\e - Y = A^\e-A$, we get
\beaa
\|Y^\e-Y\|_\dbP^2 \le C\dbE^\dbP\Big[\Big(\bigvee_0^T (A^\e-A)\Big)^2\Big] \le C\dbE^\dbP\Big[\Big(\int_0^T |a^\e_t - a_t|dK_t\Big)^2\Big] \le C\e.
\eeaa
Then
\beaa
\|Y^\e\|_\dbP^2 \le C\|Y\|_\dbP^2 + C\e.
\eeaa
Plug this and \reff{semi-est6} into \reff{semi-est5}, we get
\beaa
\dbE^\dbP\Big[ \big(\bigvee_{0}^T A \big)^2 \Big] \le C\|Y\|_{\dbP}^2 + C\dbE^\dbP [\la M\ra_T  ]   + C\e.
\eeaa
Since $\e$ is arbitrary, we obtain \reff{semi-claim}.

\ms
{\it Step 4.} We now prove \reff{semi-claim} for the general case. Since $A$ has finite variation, we can decompose $A = A^c + A^d$, where $A^c$ is the continuous part and $A^d$ is the part with pure jumps. Since $Y$ is {\cad} and $M$ is continuous, $A$ and $A^d$ are {\cad}. We denote $Y^c_t = Y_0 + M_t + A^c_t$. From {\it Step 3} we have
\beaa
\dbE^\dbP\Big[ |Y_0|^2 + \la M\ra_T + \Big(\bigvee_0^T A^c\Big)^2\Big] \le C\|Y^c\|_{\dbP}^2. 
\eeaa
Note that 
\beaa
\|Y^c\|_{\dbP} \le  \|Y\|_{\dbP} + \|A^d\|_{\dbP}
\eeaa
and apply the left inequality of \reff{semi-est} on $A^d$ we see that
\beaa
 \|A^d\|_{\dbP}^2 \le C\dbE^\dbP\Big[\Big(\bigvee_0^T A^d\Big)^2\Big].
 \eeaa
 Then
 \beaa
 \dbE^\dbP\Big[ |Y_0|^2 + \la M\ra_T + \Big(\bigvee_0^T A^c\Big)^2\Big]  \le C\|Y\|_{\dbP}^2 +  C\dbE^\dbP\Big[\Big(\bigvee_0^T A^d\Big)^2\Big].
 \eeaa
 Moreover, note that
 \beaa
 \bigvee_0^T A\le \bigvee_0^T A^c + \bigvee_0^T A^d.
 \eeaa
Thus, to prove \reff{semi-claim}, it suffices to show that
 \bea
 \label{Ad}
 \dbE^\dbP\Big[\Big(\bigvee_0^T A^d\Big)^2\Big] \le C\|Y\|_{\dbP}^2.
 \eea
 
 To this end, we first note that
 \bea
 \label{AdY}
 \bigvee_0^T A^d = \sum_{0< t\le T} |\D A_t| =  \sum_{0< t\le T} |\D Y_t|.
 \eea
 Define, for each $n$,
 \beaa
 D_n := \sum_{0< t\le T} |\D Y_t| \1_{\{|\D Y_t|\ge {1\over n}\}},
 \eeaa
 and, $\t^n_0 := 0$, and for $m\ge 0$, by denoting $Y_t:= Y_T$ for $t\ge T$,
 \beaa
 \t^n_{m+1} := \inf\Big\{ t> \t^n_m: |\D Y_t| \ge {1\over n}\Big\}\wedge (T+1).
 \eeaa
 We remark that we use $T+1$ instead of $T$ here so that $\D Y_T$ will not be counted repeatedly at below.  By the right continuity of $\dbF$ we see that $\t^n_i$ are stopping times.
 It is clear that
 \beaa
 D_n \uparrow  \sum_{0\le t\le T} |\D Y_t|~~\mbox{as}~~n\to\infty, &\mbox{and}& \sum_{i=1}^m |\D Y_{\t^n_i}| \uparrow D_n~~\mbox{as}~~m\to\infty.
 \eeaa
 Therefore, to obtain \reff{Ad} it suffices to show that
 \bea
 \label{Ad2}
 \dbE^\dbP\Big[\Big(\sum_{i=1}^m |\D Y_{\t^n_i}|\Big)^2\Big] \le \|Y\|_{\dbP}^2&\mbox{for all}~~n, m.
 \eea
 
 We now fix $n, m$. Since $\dbF$ is a Brownian filtration, all $\dbF$ - stopping time is predictable, see e.g.  $\cite{RY}$, Corollary 5.7. Then for each $\t^n_i$, there exist $\{\t^n_{i,j}, j\ge 1\}$ such that $\t^n_{i,j}<\t^n_i$ and $\t^n_{i,j}\uparrow \t^n_i$ as $j\to\infty$. By definition of $\|Y\|_\dbP$  \reff{YP}, we have
 \bea
 \label{Ad3}
  \dbE^\dbP\Big[\Big(\sum_{i=1}^m |\dbE^\dbP_{\t^n_{i-1} \vee \t^n_{i,j}}[Y_{\t^n_i}]-Y_{\t^n_{i-1} \vee \t^n_{i,j}}|\Big)^2\Big] \le  \|Y\|_{\dbP}^2.
  \eea
  Send $j\to \infty$, since $\dbF$ is continuous, we see that
  \beaa
  \lim_{j\to\infty}[\dbE^\dbP_{\t^n_{i-1} \vee \t^n_{i,j}}[Y_{\t^n_i}]-Y_{\t^n_{i-1} \vee \t^n_{i,j}}] =  Y_{\t^n_i}-Y_{\t^n_i-} =\D Y_{\t^n_i} .
  \eeaa
 Then, noting that $\dbE^\dbP[\sup_{0\le t\le T}|\D Y_t|^2] <\infty$, by applying the Dominated Convergence Theorem we obtain \reff{Ad2} from \reff{Ad3}. This implies \reff{Ad}, which in turn implies \reff{semi-claim}.
 \end{proof}	

As a direct consequence of the above a priori estimates, we have
\begin{thm}
\label{thm-semimg}
A process $Y\in \dbD(\dbF)$ is a square integrable semimartingale if and only if $\|Y\|_\dbP<\infty$.
\end{thm}
\begin{proof} By Theorem \ref{thm-semimg-est}, it suffices to prove the if part. Assume $\|Y\|_\dbP<\infty$.  From Remark \ref{rem-quasimg}, $Y$ is a quasimartingale. By Rao's theorem, see e.g. $\cite{Protter}$, Chapter III Theorem 18, $Y = M + A$, where $M$ is a local martingale and $A$ is a predictable process with paths of locally integrable variation. By the standard stopping technique  and by Theorem $\ref{thm-semimg-est}$ again, it is easy to see that $Y$ is indeed a  square-integrable semimartingale.
\end{proof}


\section{Doubly Reflected BSDEs}
\label{sect-DRBSDE}
 \setcounter{equation}{0}

In this section we assume $\dbF$ is generated by a standard Brownian motion $B$ and  augmented with all the $\dbP$-null sets. 
We consider the following Doubly Reflected Backward SDE (DRBSDE, for short) with $\dbF$-progressively measurable solution triplet $(Y, Z, A)$:
\bea
\label{2RBSDE}
\left\{\ba{lll}
\dis Y_t = \xi + \int_t^T f_s(Y_s, Z_s) ds - \int_t^T Z_s dB_s + A_T - A_t;\\
\dis L \le Y \le U,\q [Y_{t-} - L_{t-}] d K^+_t = [U_{t-} - Y_{t-}] d K^-_t=0.
\ea\right.
\eea
Here  $Y\in \dbD(\dbF)$ and $A$ has finite variation with orthogonal decomposition $A = K^+-K^-$. We say $(Y,Z,A)$ satisfying \reff{2RBSDE} is a local solution if
\bea
\label{local}
\sup_{0\le t\le T} |Y_t| + \int_0^T |Z_t|^2 dt + \bigvee_0^T A <\infty,~~\dbP\mbox{-a.s.}
\eea
and a solution if
\bea
\label{2RBSDEnorm}
\|(Y,Z,A)\|^2 :=\dbE^\dbP\Big[\sup_{0\le t\le T}|Y_t|^2 + \int_0^T |Z_t|^2dt + \big(\bigvee_0^T A \big)^2\Big]<\infty.
\eea
Throughout this section, we assume the following standing assumptions:

\begin{assum}
\label{assum-standing}
(i) $\xi$ is $\cF_T$-measurable, $f(\cd,0,0)$ is $\dbF$-progressively measurable, and 
\bea
\label{I0}
I_0^2 := I_0^2(\xi,f):= \dbE^\dbP\Big[|\xi|^2 + \big(\int_0^T|f_t(0,0)|dt\big)^2\Big] <\infty.
\eea

(ii) $f$ is uniformly Lipschitz continuous in $(y,z)$;

 (iii) $L, U \in \dbD(\dbF)$; $L\le U$, $L_T \le \xi \le U_T$; and
 \bea
 \label{LU0}
 \|(L, U)\|_{\dbP,0}^2 := \|L^+\|_{\dbP,0}^2 + \|U^-\|_{\dbP,0}^2<\infty.
 \eea
 \end{assum}

 \begin{rem}
 \label{rem-f}
{\rm  In the standard BSDE literature, one requires $\dbE^\dbP\Big[\int_0^T |f(t,0,0)|^2 dt\Big]<\infty$. Our condition \reff{I0}  is slightly weaker. In fact, most estimates in the BSDE literature can be improved by replacing $\dbE^\dbP\Big[\int_0^T |f(t,0,0)|^2 dt\Big]$ with  $ \dbE^\dbP\Big[\big(\int_0^T|f(t,0,0)|dt\big)^2\Big]$, and the arguments are rather standard. We refer to the Appendix of the monograph \cite{CZ}  for interested readers.
\qed}
 \end{rem}

It is well known that Assumption \ref{assum-standing} does not yield the wellposedness of DRBSDE \reff{2RBSDE}. At below is a simple counterexample.

\begin{eg}
\label{eg-nosolution}
Let $L=U$ be deterministic, {\cad}, and  $\dis\bigvee_0^T L = \infty$. Then DRBSDE \reff{2RBSDE} with $\xi=L_T$ and $f=0$ has no solution.
\end{eg}
\begin{proof}
Assume there is a solution $(Y, Z, A)$. Since $L\le Y\le U$,  one must have $Y = L$, which leads to $Z=0$ and $A=L$. But this contradicts with the assumption that $L$ has infinite variation.
\end{proof}

In the literature, there are two approaches for wellposedness of DRBSDEs. We first report a result from Hamadene, Hassani and Ouknine  \cite{HHO}:
\begin{lem}
\label{lem-HHO}
Let Assumption \ref{assum-standing} hold. Assume further the following separation condition:  
\bea
\label{separate}
L_t < U_t &\mbox{and}& L_{t-} < U_{t-} \q\mbox{for all}~t.
\eea
Then \reff{2RBSDE} admits a local solution.
\end{lem}
 The condition \reff{separate} is mild and easy to verify, but it does not yield any a priori estimates.  We remark that \cite{HHO} takes a slightly different form of DRBSDEs. But that is mainly for the sake of uniqueness. One can easily check that a local solution in \cite{HHO} is a local solution in our sense, so the existence in Lemma \ref{lem-HHO} is valid.

 We next report a result from Peng and Xu \cite{PX}, following the original work Cvitanic and Karatzas \cite{CK}:

\begin{lem}
\label{lem-PX}
Let Assumption \ref{assum-standing} hold. Assume further the following  Mokobodski's type of  condition:  
\bea
\label{Mokobodski}
\mbox{there exists a square integrable semimartingale  $Y^0$ such that}~ L_t \le Y^0_t \le U_t.
\eea
Then DRBSDE \reff{2RBSDE} admits a unique solution and the following estimate  holds:
\bea
\label{2RBSDE-est0}
\|(Y,Z,A)\|^2  &\le& C\Big[I_0^2 + \|Y^0\|_\dbP^2\Big].
\eea
\end{lem}

However, in those works there is no discussion  on the sufficient conditions for the existence of such $Y^0$.  Our goal in this section is to provide a tractable equivalent condition. In light of the norm $\|.\|_\dbP$ \reff{YP}, we introduce the following norm for the barriers $(L, U)$:
\bea
 \label{norm-LU}
 \|(L, U)\|_{\dbP}^2 &:=&  \|(L, U)\|_{\dbP,0}^2\\
 &+&\sup_\pi \dbE^\dbP\Big[\Big(\sum_{i=0}^{n-1} \big([\dbE^\dbP_{\t_i}(L_{\t_{i+1}})  - U_{\t_i}]^++[L_{\t_i} - \dbE^\dbP_{\t_i}(U_{\t_{i+1}})]^+\big)\Big)^2\Big],\nonumber
 \eea
 where the supremum is again taken over all partitions $\pi: 0=\t_0\le \cds\le\t_n=T$. Our main result of this section is:
 
 \begin{thm}
 \label{thm-DRBSDE}
 Let Assumption \ref{assum-standing} hold. Then the following are equivalent:
 
 (i) The DRBSDE \reff{2RBSDE} admits a unique solution $(Y,Z,A)$;
 
 (ii) the Mokobodski condition \reff{Mokobodski} holds;
 
 (iii)  $\|(L, U)\|_{\dbP}<\infty$.
 
 \no Moreover, in this case we have the estimate:
 \bea
\label{2RBSDE-est1}
\|(Y,Z,A)\|^2 &\le& C\Big[I_0^2 +\|(L, U)\|_{\dbP}^2\Big].
\eea
\end{thm}

In addition, we have the following estimates for the difference of two DRBSDEs:

\begin{thm}
\label{thm-difference}
Assume $(\xi_i, f^i, L^i, U^i)$, $i=1,2$, satisfy all the conditions in Theorem \ref{thm-DRBSDE}, and let $(Y^i, Z^i, A^i)$ denote the solution to the corresponding DRBSDE \reff{2RBSDE}. Denote $\d Y:= Y^1-Y^2$, and similarly for the other notations. Then
\bea
\label{norm-est2}
&&\dbE^\dbP\Big[\sup_{0\le t\le T}[|\d Y_t|^2 + |\d A_t|^2] + \int_0^T |\d Z_t|^2dt\Big]\nonumber \\
&\le& C\dbE^\dbP\Big[|\d \xi|^2 + \Big(\int_0^T |\d f (t, Y^1_t, Z^1_t)|dt\Big)^2\Big]\\
&&+C\sum_{i=1}^2\Big[ I_0(\xi_i, f^i) + \|(L^i, U^i)\|_\dbP\Big]  \Big(\dbE^\dbP\Big[\sup_{0\le t\le T}[|\d L_t|^2 + |\d U_t|^2]\Big] \Big)^{1\over 2} .\nonumber
\eea
\end{thm}
These two theorems will be proved in the rest of this section. We first note that

 \begin{rem}
 \label{rem-onebarrier}
 {\rm (i) In the case that there is only one barrier $L$, we may view it as $U = \infty$. One can check straightforwardly that $\|(L, U)\|_{\dbP} = \|L^+\|_{\dbP,0}$. Then Theorems \ref{thm-DRBSDE} and \ref{thm-difference} reduce to standard results for reflected BSDEs with one barrier, see El Karoui et al \cite{EKPPQ}. 
 
 (ii) In the case $(L^1, U^1) = (L^2, U^2)$, the last term in \reff{norm-est2} vanishes and  Peng and Xu \cite{PX} has already obtained the estimate. 
  \qed}
 \end{rem}

\subsection{Proof of Theorem \ref{thm-difference}.} 

As usual we start with some   a priori estimates.
\begin{lem}
 \label{lem-difference}
 Assume $(\xi_i, f^i, L^i, U^i)$, $i=1,2$, satisfy Assumption \ref{assum-standing}.  If the corresponding DRBSDE \reff{2RBSDE} has a solution $(Y^i, Z^i, A^i)$, then
 \bea
 \label{norm-est3}
 \dbE^\dbP\Big[\sup_{0\le t\le T} [|\d Y_t|^2 + |\d A_t|^2] + \int_0^T |\d Z_t|^2dt\Big]\le CI^2,
 \eea
 where, recalling the norm $\|(Y,Z,A)\|$ defined by \reff{2RBSDEnorm},
 \bea
 \label{I2}
I^2 &:=& \dbE^\dbP\Big[|\d \xi|^2 + \Big(\int_0^T |\d f(t, Y^1_t, Z^1_t)|dt\Big)^2\Big] \\
&+&\sum_{i=1}^2\|(Y^i, Z^i, A^i)\|  \Big(\dbE^\dbP\Big[\sup_{0\le t\le T}[|\d L_t|^2 + |\d U_t|^2]\Big] \Big)^{1\over 2} .\nonumber
 \eea
\end{lem}

\begin{proof} Let $\l>0$ be a constant which will be specified later. Applying It\^{o}'s formula on $e^{\l t}|\d Y_t|^2$ we have
\bea
\label{diff-est1}
&& e^{\l t}|\d Y_t|^2 + \l\int_t^T e^{\l s} |\d Y_s|^2 ds + \int_t^T e^{\l s}| \d Z_s|^2 ds\\
 &=& e^{\l T}|\d \xi^2| + 2\int_t^T e^{\l s}\d Y_s \big( f^1(s,Y^1_s,Z^1_s) - f^2(s,Y^2_s,Z^2_s)\big)ds + 2\int_t^T e^{\l s}\d Y_{s-} d\d A_s \nonumber\\
 &&- 2\int_t^T e^{\l s}\d Y_s \d Z_s dB_s.\nonumber
\eea
For any $\e>0$, note that
\bea
\label{diff-est2}
&&  2\int_t^T e^{\l s}|\d Y_s| \big| f^1(s,Y^1_s,Z^1_s) - f^2(s,Y^2_s,Z^2_s)\big|ds \nonumber\\
&\le& C\int_t^T e^{\l s}|\d Y_s|[|\d f (s,Y^1_s,Z^1_s)| + |\d Y_s| + |\d Z_s|] ds\nonumber\\
&\le& C\Big[\sup_{t\le s\le T}|\d Y_s| \int_t^T e^{\l s} |\d f (s,Y^1_s,Z^1_s)|ds + \int_t^T e^{\l s}[|\d Y_s|^2 + |\d Y_s||\d Z_s|] ds\Big]\nonumber\\
&\le& \e \sup_{t\le s\le T}|\d Y_s| ^2 + {1\over 2}\int_t^T e^{\l s} |\d Z_s|^2 ds  \\
&&\qq + C\int_t^T e^{\l s} |\d Y_s|^2 ds+ C\e^{-1}\big(\int_t^T e^{\l s} |\d f (s,Y^1_s,Z^1_s)|ds \big)^2;\nonumber
\eea
and, with the orthogonal decompositions $A^i = K^{i,+}-K^{i-}$,
\bea
\label{diff-est3}
&& 2\int_t^T e^{\l s}\d Y_{s-} d\d A_s\nonumber\\
&=& 2\int_t^T e^{\l s}\big(Y^1_{s-} dK^{1,+}_s  - Y^1_{s-} dK^{1,-}_s - Y^2_{s-} dK^{1,+}_s + Y^2_{s-} dK^{1,-}_s\nonumber \\
&&\qq - Y^1_{s-} dK^{2,+}_s + Y^1_{s-} dK^{2,-}_s + Y^2_{s-} dK^{2,+}_s - Y^2_{s-}d K^{2,-}_s \big)\nonumber\\
&\le&  2\int_t^T e^{\l s}\big(L^1_{s-} dK^{1,+}_s  - U^1_{s-} dK^{1,-}_s - L^2_{s-} dK^{1,+}_s + U^2_{s-} dK^{1,-}_s\nonumber \\
&&\qq - L^1_{s-} dK^{2,+}_s + U^1_{s-} dK^{2,-}_s + L^2_{s-} dK^{2,+}_s - U^2_{s-}d K^{2,-}_s \big)\nonumber\\
&=& 2\int_t^T e^{\l s}\big(\d L_{s-} dK^{1,+}_s  - \d U_{s-} dK^{1,-}_s    -\d L_{s-} dK^{2,+}_s + \d U_{s-} dK^{2,-}_s \big)\nonumber\\
&\le& 2e^{\l (T-t)}\sup_{0\le s\le T} [|\d L_s| + |\d U_s|] \big[\bigvee_t^T A^1 + \bigvee_t^T A^2\big] .
\eea
Plug \reff{diff-est2} and \reff{diff-est3} into \reff{diff-est1}, we obtain
\beaa
&& e^{\l t}|\d Y_t|^2 + \l\int_t^T e^{\l s} |\d Y_s|^2 ds + \int_t^T e^{\l s}| \d Z_s|^2 ds\nonumber\\
 &\le& e^{\l T}|\d \xi^2| + \e \sup_{t\le s\le T}|\d Y_s| ^2 + {1\over 2}\int_t^T e^{\l s} |\d Z_s|^2 ds  \\
&& + C\int_t^T e^{\l s} |\d Y_s|^2 ds+ C\e^{-1}\big(\int_t^T e^{\l s} |\d f (s,Y^1_s,Z^1_s)|ds \big)^2\nonumber\\
&&+ 2e^{\l (T-t)}\sup_{0\le s\le T} [|\d L_s| + |\d U_s|] \big[\bigvee_t^T A^1 + \bigvee_t^T A^2\big] - 2\int_t^T e^{\l s}\d Y_s \d Z_s dB_s.\nonumber
\eeaa
Set $\l =C$ for the above $C$, we get
\bea
\label{diff-est4}
&& e^{\l t}|\d Y_t|^2 + {1\over 2}\int_t^T e^{\l s}| \d Z_s|^2 ds\nonumber\\
 &\le& e^{\l T}|\d \xi^2| + \e \sup_{t\le s\le T}|\d Y_s| ^2 + C\e^{-1}\big(\int_t^T e^{\l s} |\d f (s,Y^1_s,Z^1_s)|ds \big)^2\\
&&+ 2e^{\l (T-t)}\sup_{0\le s\le T} [|\d L_s| + |\d U_s|] \big[\bigvee_t^T A^1 + \bigvee_t^T A^2\big] - 2\int_t^T e^{\l s}\d Y_s \d Z_s dB_s.\nonumber
\eea
Take expectation on both sides, we have
\bea
\label{diff-est5}
\sup_{0\le t\le T} \dbE^\dbP[|\d Y_t|^2] + \dbE^\dbP\Big[ \int_0^T | \d Z_t|^2 dt\Big]\leq C[1+\e^{-1}]I^2 + \e \dbE^\dbP\Big[\sup_{0\le t\le T}|\d Y_t| ^2\Big].
 \eea
Moreover, by \reff{diff-est4} we have
\bea
\label{diff-est6}
&&\sup_{0\le t\le T} e^{\l t}|\d Y_t|^2 \nonumber\\
 &\le& e^{\l T}|\d \xi^2| + \e \sup_{0\le t\le T}|\d Y_t| ^2 + C\e^{-1}\big(\int_0^T e^{\l t} |\d f (t,Y^1_t,Z^1_t)|dt \big)^2\\
&&+ 2e^{\l T}\sup_{0\le t\le T} [|\d L_t| + |\d U_t|] \big[\bigvee_0^T A^1 + \bigvee_0^T A^2\big] +2\sup_{0\le t\le T}\Big|\int_t^T e^{\l s}\d Y_s \d Z_s dB_s\Big|.\nonumber
\eea
Apply the Burkholder-Davis-Gundy Inequality and note that $\l=C$,  we get
\bea
\label{diff-est7}
&&\dbE^\dbP\Big[\sup_{0\le t\le T}\Big|\int_t^T e^{\l s}\d Y_s \d Z_s dB_s\Big|\Big] \le C\dbE^\dbP\Big[\Big(\int_0^T |\d Y_t \d Z_t|^2 dt\Big)^{1\over 2}\Big]\nonumber\\
&&\qq\qq \le C\dbE^\dbP\Big[\sup_{0\le t\le T}|\d Y_t| \Big(\int_0^T |\d Z_t|^2 dt\Big)^{1\over 2}\Big]\\
&&\qq\qq \le \sqrt{\e} \dbE^\dbP\Big[\sup_{0\le t\le T}|\d Y_t|^2\Big] + C\e^{-{1\over 2}} \dbE^\dbP\Big[\int_0^T |\d Z_t|^2 dt\Big].\nonumber
\eea
Take expectation on both sides of \reff{diff-est6},  and apply \reff{diff-est7} and then \reff{diff-est5}, we obtain
\beaa
\dbE^\dbP\Big[\sup_{0\le t\le T} |\d Y_t|^2\Big]&\le& C[1+\e^{-1}]I^2 + C\e \dbE^\dbP\Big[\sup_{0\le t\le T}|\d Y_t| ^2\Big]\\
&& + C\sqrt{\e} \dbE^\dbP\Big[\sup_{0\le t\le T}|\d Y_t|^2\Big] + C\e^{-{1\over 2}}\dbE^\dbP\Big[\int_0^T |\d Z_t|^2 dt\Big]\\
&\le& C\big[\sqrt{\e} + \e(1+\e^{-{1\over 2}})\big] \dbE^\dbP\Big[\sup_{0\le t\le T}|\d Y_t|^2\Big] + C[1+\e^{-{1\over 2}}][1+\e^{-1}]I^2\\
&\le& C\sqrt{\e} \dbE^\dbP\Big[\sup_{0\le t\le T}|\d Y_t|^2\Big] + C\e^{-{3\over 2}}I^2
\eeaa
Set $\e:= {1\over 4C^2}$ for the above $C$. Then 
\beaa
\dbE^\dbP\Big[\sup_{0\le t\le T} |\d Y_t|^2\Big]\le CI^2 .
\eeaa
Plug this into \reff{diff-est5}, we get
\beaa
\dbE^\dbP\Big[\int_0^T |\d Z_t|^2dt\Big]\le CI^2 .
\eeaa
Finally, notice that
\beaa
\d A_t = \d Y_0 - \d Y_t - \int_0^t [f^1(s, Y^1_s, Z^1_s) - f^2(s, Y^2_s, Z^2_s)] ds + \int_0^t \d Z_s dB_s.
\eeaa
One can easily get the estimate for $\d A$.
\end{proof}

\no {\bf Proof of Theorem \ref{thm-difference}.} This is a direct consequence of Lemma \ref{lem-difference} and Theorem \ref{thm-DRBSDE}.\qed

\ms

We emphasize that in next subsection, we shall prove Theorem \ref{thm-DRBSDE} by using Lemma \ref{lem-difference}, but without using Theorem \ref{thm-difference}. So there is no danger of cycle proof.

\subsection{Proof of Theorem \ref{thm-DRBSDE}}

Again, we start with a priori estimate. Recall Lemma \ref{lem-HHO}.
\begin{lem}
\label{lem-apriori}
Let Assumption \ref{assum-standing}  and \reff{separate} hold, and $f=0$. Then the local solution $(Y,Z,A)$ of DRBSDE \reff{2RBSDE} satisfies  \reff{2RBSDE-est1}.
\end{lem}
\begin{proof}  Without loss of generality, we assume $\|(L, U)\|_{\dbP} < \infty$.  We proceed in three steps.

\ms
{\it Step 1.} We first assume $(Y, Z, A)$ is a solution of \reff{2RBSDE} and $Y$ is continuous. Then   $K^+$ and $K^-$ are also continuous. Apply It\^{o}'s formula on $|Y_t|^2$, by the minimum condition in \reff{2RBSDE} we have,
\bea
\label{RY2}
d |Y_t|^2 &=& 2Y_t Z_t dB_t + |Z_t|^2 dt -2Y_{t-} dK^+_t + 2 Y_{t-} dK^-_t \nonumber\\
&=& 2Y_t Z_t dB_t + |Z_t|^2 dt  - 2L_{t-} d K^+_t + 2 U_{t-} dK^-_t .
\eea
Then, for any $\e>0$,
\beaa
&&\dbE^\dbP\Big[|Y_t|^2 + \int_t^T |Z_s|^2 ds\Big] = \dbE^\dbP\Big[|\xi|^2 + 2\int_t^T L_{s-} dK^+_s - 2\int_t^T U_{s-} dK^-_s\Big] \\
&\le& \dbE^\dbP\Big[|\xi|^2 + 2\sup_{0\le s\le T} L^+_s K^+_T + 2\sup_{0\le  s\le T} U^-_s K^-_T \Big] \\
&\le& \dbE^\dbP\Big[|\xi|^2 + C\e^{-1} \sup_{0\le s\le T} [|L^+_s|^2 + |U^-_s|^2] + \e[ |K^+_T|^2 +|K^-_T|^2] \Big]\\
&\le& \dbE^\dbP[\xi|^2] + C\e^{-1} \|(L, U)\|_{\dbP,0}^2  + \e   \dbE^\dbP\Big[ \big(\bigvee_0^T A\big)^2 \Big].
\eeaa
Following standard arguments, in particular by applying the Burkholder-Davis-Gundy Inequality on \reff{RY2}, we have
\bea
\label{EY2}
\dbE^\dbP\Big[\sup_{0\le t\le T}|Y_t|^2 + \int_0^T |Z_t|^2 dt\Big] \le  C\dbE^\dbP[\xi|^2] + C\e^{-1} \|(L, U)\|_{\dbP,0}^2 + C\e   \dbE^\dbP\Big[ \big(\bigvee_0^T A\big)^2\Big].
\eea
We claim that
\bea
\label{EK2}
\dbE^\dbP\Big[\big(\bigvee_0^T A\big)^2 \Big] \le C\dbE^\dbP\Big[\sup_{0\le t\le T}|Y_t|^2 + \int_0^T |Z_t|^2 dt\Big] + C\|(L, U)\|_{\dbP}^2.
\eea
Combine \reff{EY2} and \reff{EK2} and set $\e$ small, we prove \reff{2RBSDE-est1} immediately.

To prove \reff{EK2}, we define a sequence of stopping times:  $\t_0:=0$ and, for $i\ge 0$,
\bea
\label{taui}
\left.\ba{lll}
 \t_{2i+1} &:=& \inf \{ t \ge \t_{2i}:  K^+_t > K^+_{\t_{2i}} \} \wedge T,  \\
  \t_{2i+2} &:=& \inf \{ t \ge \t_{2i+1}: K^-_t > K^-_{\t_{2i+1}} \} \wedge T.
 \ea\right.
 \eea 
 Then  $d K^+_t =0$ on $[\t_{2i}, \t_{2i+1}]$ and $dK^-_t=0$ on $[\t_{2i+1}, \t_{2i+2}]$, and thus
\bea
\label{Ytaui}
\left.\ba{lll}
\dis  Y_t= Y_{\t_{2i+1}}  - \int_t^{\t_{2i+1}} Z_sdB_s - (K^-_{\t_{2i+1}} - K^-_t),\q t\in [\t_{2i}, \t_{2i+1}];\\
\dis  Y_t= Y_{\t_{2i+2}}  - \int_t^{\t_{2i+2}} Z_sdB_s + (K^+_{\t_{2i+2}} - K^+_t),\q t\in [\t_{2i+1}, \t_{2i+2}];
\ea\right.
\eea
Since $L$ and $U$ are right continuous and $K$ is continuous,  by the minimum condition in \reff{2RBSDE} we have
\bea
\label{Y=LU}
Y_{\t_{2i}} = U_{\t_{2i}}\1_{\{\t_{2i} <T\}} + \xi \1_{\{\t_{2i} = T\}}  \q\mbox{and}\q Y_{\t_{2i+1}} = L_{\t_{2i+1}}\1_{\{\t_{2i+1} <T\}} + \xi \1_{\{\t_{2i+1} = T\}}.
\eea
In particular, on $\{\t_{2i}<T\}$, we have $Y_{\t_{2i}} = U_{\t_{2i}} > L_{\t_{2i}}$, then $Y_t > L_t$ for $t$ in a right neighborhood of $\t_{2i}$ and thus $dK^+_t = 0$. This implies that $\t_{2i+1} > \t_{2i}$ on $\{\t_{2i}<T\}$. Similarly, $\t_{2i+2} > \t_{2i+1}$ on $\{\t_{2i+1}<T\}$. Moreover, as in \cite{HHO}, we see that
\bea
\label{taulimit}
\mbox{for a.s. $\o$, $\t_n(\o) = T$ for $n$ large enough.}
\eea
Indeed, denote $\t^* := \lim_{n\to\infty} \t_n$. If $\t^*<T$, then $\t_n < \t^*$ for all $n$ and we get
\beaa
L_{\t^*-} = \lim_{n\to\infty} L_{\t_{2i+1}} =   \lim_{n\to\infty} Y_{\t_{2i+1}} = Y_{\t^*-} = \lim_{n\to\infty} Y_{\t_{2i}} = \lim_{n\to\infty} U_{\t_{2i}} = U_{\t^*-}.
\eeaa
This contradicts with \reff{separate}.

For each $i$, by \reff{Ytaui} and \reff{Y=LU},
\beaa
0&\le& \dbE^\dbP_{\t_{2i}} [K^-_{\t_{2i+1}}] - K^-_{\t_{2i}} = \dbE^\dbP_{\t_{2i}} [Y_{\t_{2i+1}}] - Y_{\t_{2i}}\\
 &=& \dbE^\dbP_{\t_{2i}} \Big[L_{\t_{2i+1}}\1_{\{\t_{2i+1} <T\}} + \xi \1_{\{\t_{2i+1} = T\}}\Big]-U_{\t_{2i}}\1_{\{\t_{2i} <T\}} - \xi \1_{\{\t_{2i} = T\}}\\
 &=& \Big[\dbE^\dbP_{\t_{2i}} [L_{\t_{2i+1}}] - U_{\t_{2i}}\Big] \1_{\{\t_{2i} <T\}} + \dbE^\dbP_{\t_{2i}} \Big[[L_{\t_{2i+1}}-\xi]\1_{\{\t_{2i}<T=\t_{2i+1}\}} \Big]\\
 &\le& \Big[\dbE^\dbP_{\t_{2i}} [L_{\t_{2i+1}}] - U_{\t_{2i}}\Big]^+
\eeaa
Then for any $n$,
\beaa
\dbE^\dbP\Big[ \Big ( \sum_{i=0}^n \big[\dbE^\dbP_{\t_{2i}} [K^-_{\t_{2i+1}}] - K^-_{\t_{2i}}\big] \Big ) \Big]^2 \le \|(L, U)\|_{\dbP}^2.
\eeaa
Send $n\to\infty$, we get
\bea
\label{DK-2}
\dbE^\dbP\Big[ \Big ( \sum_{i\ge 0} \big[\dbE^\dbP_{\t_{2i}} [K^-_{\t_{2i+1}}] - K^-_{\t_{2i}}\big] \Big ) \Big]^2 \le \|(L, U)\|_{\dbP}^2.
\eea
Similarly,
\bea
\label{DK+2}
\dbE^\dbP\Big[ \Big ( \sum_{i\ge 0} \big[\dbE^\dbP_{\t_{2i+1}} [K^+_{\t_{2i+2}}] - K^+_{\t_{2i+1}}\big] \Big ) \Big]^2 \le \|(L, U)\|_{\dbP}^2.
\eea

Denote
\bea
\label{hatYn}
\hat{Y}_{\t_n} := Y_{\t_n} - \sum_{i\le {n\over 2}}\big[\dbE^\dbP_{\t_{2i}} [K^-_{\t_{2i+1}}] - K^-_{\t_{2i}}\big] + \sum_{i\le {n-1\over 2}} \big[\dbE^\dbP_{\t_{2i+1}} [K^+_{\t_{2i+2}}] - K^+_{\t_{2i+1}}\big] .
\eea
By \reff{DK-2} and \reff{DK+2}, we have
 \bea
 \label{EhatY}
 \dbE^\dbP\Big[\max_{n\ge 0}|\hat Y_{\t_n}|^2\Big] \le  C\dbE^\dbP\Big[\sup_{0\le t\le T}|Y_t|^2\Big] +C \|(L, U)\|_{\dbP}^2.
 \eea
 Note that
\beaa
\hat Y_{\t_n}= Y_0 + \int_0^{\t_n} Z_s dB_s + \sum_{i\le {n\over 2}}\big[K^-_{\t_{2i+1}} - \dbE^\dbP_{\t_{2i}} [K^-_{\t_{2i+1}}] \big] -\sum_{i\le {n-1\over 2}} \big[K^+_{\t_{2i+2}}- \dbE^\dbP_{\t_{2i+1}} [K^+_{\t_{2i+2}}] \big]
\eeaa
is a martingale.  By \reff{EhatY}, we have
\beaa
&&\dbE^\dbP\Big[\sum_{i\le {n\over 2}}\big[K^-_{\t_{2i+1}} - \dbE^\dbP_{\t_{2i}} [K^-_{\t_{2i+1}}] \big]^2 +\sum_{i\le {n-1\over 2}} \big[K^+_{\t_{2i+2}}- \dbE^\dbP_{\t_{2i+1}} [K^+_{\t_{2i+2}}] \big] ^2\Big]\\
&=& \dbE^\dbP\Big[\Big(\sum_{i\le {n\over 2}}\big[K^-_{\t_{2i+1}} - \dbE^\dbP_{\t_{2i}} [K^-_{\t_{2i+1}}] \big] -\sum_{i\le {n-1\over 2}} \big[K^+_{\t_{2i+2}}- \dbE^\dbP_{\t_{2i+1}} [K^+_{\t_{2i+2}}] \big] \Big)^2\Big]\\
&=& \dbE^\dbP\Big[\Big(\hat Y_{\t_n}- Y_0 - \int_0^{\t_n} Z_s dB_s \Big)^2\Big] \le C\dbE^\dbP\Big[\sup_{i\ge 0}|\hat Y_{\t_i}|^2 + \int_0^{\t_n}|Z_t|^2dt\Big]\\
&\le&   C\dbE^\dbP\Big[\sup_{0\le t\le T}|Y_t|^2+\int_0^T |Z_t|^2dt\Big]+C \|(L, U)\|_{\dbP}^2.
\eeaa
Send $n\to\infty$ and
\bea
\label{DK2}
&&\dbE^\dbP\Big[\sum_{i\ge 0}\big[K^-_{\t_{2i+1}} - \dbE^\dbP_{\t_{2i}} [K^-_{\t_{2i+1}}] \big]^2 +\sum_{i\ge 0} \big[K^+_{\t_{2i+2}}- \dbE^\dbP_{\t_{2i+1}} [K^+_{\t_{2i+2}}] \big] ^2\Big]\nonumber\\
&=& \dbE^\dbP\Big[\Big(\sum_{i\ge 0}\big[K^-_{\t_{2i+1}} - \dbE^\dbP_{\t_{2i}} [K^-_{\t_{2i+1}}] \big] -\sum_{i\ge 0} \big[K^+_{\t_{2i+2}}- \dbE^\dbP_{\t_{2i+1}} [K^+_{\t_{2i+2}}] \big] \Big)^2\Big]\nonumber\\
&\le& C\dbE^\dbP\Big[\sup_{0\le t\le T}|Y_t|^2+\int_0^T |Z_t|^2dt\Big]+C \|(L, U)\|_{\dbP}^2.
\eea
This, together with \reff{taulimit}, \reff{DK-2} and \reff{DK+2}, implies further that
\beaa
&&\dbE^\dbP\Big[|K^+_T|^2 + |K^-_T|^2\Big] = \dbE^\dbP\Big[ \Big(\sum_{i\ge 0}\big[K^-_{\t_{2i+1}} - K^-_{\t_{2i}}\big]\Big)^2 +\Big(\sum_{i\ge 0} \big[K^+_{\t_{2i+2}}-K^+_{\t_{2i+1}}\big]\Big) ^2\Big]\nonumber\\
&\le& C\dbE^\dbP\Big[\Big(\sum_{i\ge 0}\big[K^-_{\t_{2i+1}} - \dbE^\dbP_{\t_{2i}} [K^-_{\t_{2i+1}}] \big]\Big)^2 +\Big(\sum_{i\ge 0}\big[\dbE^\dbP_{\t_{2i}} [K^-_{\t_{2i+1}}- K^-_{\t_{2i}}] \big]\Big)^2\Big]\nonumber\\
&&+ C\dbE^\dbP\Big[\Big(\sum_{i\ge 0} \big[K^+_{\t_{2i+2}}- \dbE^\dbP_{\t_{2i+1}} [K^+_{\t_{2i+2}}] \big]\Big)^2 +\Big(\sum_{i\ge 0} \big[\dbE^\dbP_{\t_{2i+1}} [K^+_{\t_{2i+2}}]- K^+_{\t_{2i+1}} \big]\Big)^2  \Big]\nonumber\\
&=&C\dbE^\dbP\Big[\sum_{i\ge 0}\big[K^-_{\t_{2i+1}} - \dbE^\dbP_{\t_{2i}} [K^-_{\t_{2i+1}}] \big]^2 + \sum_{i\ge 0} \big[K^+_{\t_{2i+2}}- \dbE^\dbP_{\t_{2i+1}} [K^+_{\t_{2i+2}}] \big]^2\Big]\nonumber\\
&&+ C\dbE^\dbP\Big[\Big(\sum_{i\ge 0}\big[\dbE^\dbP_{\t_{2i}} [K^-_{\t_{2i+1}}- K^-_{\t_{2i}}] \big]\Big)^2 + \Big(\sum_{i\ge 0} \big[\dbE^\dbP_{\t_{2i+1}} [K^+_{\t_{2i+2}}]- K^+_{\t_{2i+1}} \big]\Big)^2 \Big]\nonumber\\
&\le& C\dbE^\dbP\Big[\sup_{0\le t\le T}|Y_t|^2+\int_0^T |Z_t|^2dt\Big] +C \|(L, U)\|_{\dbP}^2.
\eeaa
This proves \reff{EK2} and hence \reff{2RBSDE-est1}.

\ms

{\it Step 2.} We next assume $(Y, Z, A)$ is a local solution but $Y$ is still continuous.  Let $\t_i$ be defined by \reff{taui}. Then \reff{Ytaui}-\reff{taulimit} still hold. This implies
\beaa
Y_{\t_{2i}} &\ge&  -U^-_{\t_{2i}}\1_{\{\t_{2i} <T\}} - |\xi| \1_{\{\t_{2i} = T\}}\ge -\Big[\sup_{0\le t\le T} U_t^- + |\xi|\Big],\\
Y_{\t_{2i}} &\le& \dbE^\dbP_{\t_{2i}}[Y_{\t_{2i+1}}] \le \dbE^\dbP_{\t_{2i}}\Big[L^+_{\t_{2i+1}}\1_{\{\t_{2i+1} <T\}} + |\xi| \1_{\{\t_{2i+1} = T\}}\Big] \le \dbE^\dbP_{\t_{2i}}\Big[\sup_{0\le t\le T} L_t^+ + |\xi|\Big].
\eeaa
Then
\beaa
\max_{i\ge 0} |Y_{\t_{2i}}| &\le& \Big[\sup_{0\le t\le T} U_t^- + |\xi|\Big] \bigvee  \sup_{0\le s\le T}\dbE^\dbP_s\Big[\sup_{0\le t\le T} L_t^+ + |\xi|\Big] \\
&\le& \sup_{0\le s\le T} \dbE^\dbP_s\Big[\sup_{0\le t\le T} [L_t^+ +U_t^-]+ |\xi|\Big].
\eeaa
Thus
\bea
\label{maxYtau}
&&\dbE^\dbP\Big[\max_{i\ge 0} |Y_{\t_{2i}}|^2\Big] \le \dbE^\dbP\Big[\Big(\sup_{0\le s\le T} \dbE^\dbP_s\big[\sup_{0\le t\le T} [L_t^+ +U_t^-]+ |\xi|\big]\Big)^2\Big]\nonumber\\
&\le& C\dbE^\dbP\Big[\Big(\sup_{0\le t\le T} [L_t^+ +U_t^-]+ |\xi|\Big)^2\Big]\le C\dbE^\dbP[|\xi|^2] + C\|(L, U)\|_{\dbP,0}^2.
\eea

Now for any $n$, define
\bea
\label{hattau}
\hat\t_n :=\inf\Big\{ t: \sup_{0\le s \le t} |Y_s| + \int_0^t |Z_s|^2 ds + \bigvee_0^t A \ge n\Big\}\wedge T.
\eea
Then
\bea
\label{Yhattau}
\dbE^\dbP\Big[\sup_{0\le t < \hat\t_n} |Y_t|^2 + \int_0^{\hat\t_n} |Z_t|^2 dt + \Big(\bigvee_0^{\hat\t_n} A\Big)^2 \Big\} <\infty.
\eea
Define
\beaa
\tilde\t_n := \inf\{ \t_{2i}: \t_{2i}\ge \hat{\t}_n\}.
\eeaa
Then by \reff{2RBSDE} and \reff{maxYtau} we have
\beaa
&&Y_t = Y_{\tilde\t_n} + \int_t^{\tilde\t_n} Z_s dB_s  - (K^-_{\tilde\t_n} - K^-_t),~Y_t \le U_t, ~ [U_t - Y_t] dK^-_t = 0,~  t\in [\hat\t_n, \tilde\t_n];\\
&& \dbE^\dbP[|Y_{\tilde\t_n}|^2] \le C\dbE^\dbP[|\xi|^2] + C\|(L, U)\|_{\dbP,0}^2.
\eeaa
By standard arguments for Reflected BSDEs with one barrier, see e.g. \cite{EKPPQ}, 
\beaa
 \dbE^\dbP[|Y_{\hat\t_n}|^2] \le C\dbE^\dbP[|\xi|^2] + C\|(L, U)\|_{\dbP,0}^2 + C\dbE^\dbP\Big[\sup_{0\le t\le T}|U^-_t|^2\Big] \le C\dbE^\dbP[|\xi|^2] + C\|(L, U)\|_{\dbP,0}^2.
 \eeaa
 This, together with \reff{Yhattau}, implies that
 \beaa
 \dbE^\dbP\Big[\sup_{0\le t \le \hat\t_n} |Y_t|^2 + \int_0^{\hat\t_n} |Z_t|^2 dt + \Big(\bigvee_0^{\hat \t_n} A\Big)^2 \Big\} <\infty.
\eeaa
Then by Step 1, we obtain
\beaa
 \dbE^\dbP\Big[\sup_{0\le t \le \hat\t_n} |Y_t|^2 + \int_0^{\hat\t_n} |Z_t|^2 dt + \Big(\bigvee_0^{\hat \t_n} A\Big)^2 \Big\} &\le& C\dbE^\dbP[|Y_{\hat\t_n}|^2] + C\|(L, U)\|_{\dbP}^2\\
 &\le& C\dbE^\dbP[|\xi|^2] + C\|(L, U)\|_{\dbP}^2.
\eeaa
Note that $\hat \t_n = T$ when $n$ is large enough. Send $n\to\infty$ and apply the Monotone Convergence Theorem, we prove  \reff{2RBSDE-est1}.

\ms
{\it Step 3.}  Finally we allow $Y$ to be discontinuous.. Let
\beaa
&\dis\bar Y_t := Y_t - \sum_{0< s\le t} \D Y_s,\q \bar K^+_t := K^+_t - \sum_{0< s \le t} \D K^+_s,\q \bar K^-_t := K^-_t - \sum_{0< s \le t} \D K^-_s,&\\
&\dis\bar A_t := \bar K^+_t - \bar K^-_t,~ \bar L_t := L_t - \sum_{0< s \le t} \D K^+_s,~ \bar U_t := U_t + \sum_{0< s \le t} \D K^-_s,~ \bar\xi := \xi - \sum_{0< s\le T} \D Y_s.&
\eeaa
Then it is clear that $\bar Y$ is continuous,  $(\bar L, \bar U)$ satisfies \reff{separate}, and $(\bar Y, Z, \bar A)$ is a local solution to DRBSDE \reff{2RBSDE} with coefficients $(\bar\xi, 0, \bar L, \bar U)$.  By Step 2, we have
\beaa
\|(\bar Y, Z, \bar A)\|^2 \le C \dbE^\dbP[|\bar\xi|^2] + C \|(\bar L, \bar U)\|_{\dbP}^2.
\eeaa
One can check straightforwardly that
\beaa
&&\|(Y, Z,  A)\|^2  \le C \|(\bar Y, Z, \bar A)\|^2 + C \dbE^\dbP\Big[ \Big( \sum_{0\le t\le T} [\D K^+_t + \D K^-_t]\Big)^2\Big] ;\\
 && \dbE^\dbP[|\bar\xi |^2] \le C\dbE^\dbP[|\xi|^2] +C \dbE^\dbP\Big[ \Big( \sum_{0\le t\le T} [\D K^+_t + \D K^-_t]\Big)^2\Big];\\
 && \|(\bar L, \bar U)\|_{\dbP}^2 \le  \|(L, U)\|_{\dbP}^2 
 \eeaa
 Then
 \bea
\label{barYest}
\|(Y, Z, A)\|^2 \le C \dbE^\dbP[|\xi|^2] + C \|(L, U)\|_{\dbP}^2 + C \dbE^\dbP\Big[ \Big( \sum_{0\le t\le T} [\D K^+_t + \D K^-_t]\Big)^2\Big].
\eea

Note that, when $\D K^+_t >0$, by the minimum condition of \reff{2RBSDE} we see that $Y_{t-} = L_{t-}$. Since $K^+$ and $K^-$ are orthogonal, we have $\D Y_t = -\D K^+_t$. Thus $L_t \le Y_t = Y_{t-} - \D K^+_t  = L_{t-} - \D K^+_t$. This implies that  $\sum_{0< t \le T} \D K^+_t \le \sum_{0< t\le T} [\D L_t]^-$. Similarly we have $\sum_{0< t \le T} \D K^-_t \le \sum_{0< t\le T} [\D U_t]^+$. Following the arguments for \reff{Ad}, one can easily prove that
\beaa
\dbE^\dbP\Big[ \Big( \sum_{0< t\le T} [[\D L_t]^-  + [\D U_t]^+]\Big)^2\Big] \le C\|(L, U)\|_{\dbP}^2.
\eeaa
Then \reff{2RBSDE-est1} follows from \reff{barYest} immediately.
\end{proof}

\no{\bf Proofs of Theorem \ref{thm-DRBSDE}.} First, by Lemma \ref{lem-PX} we know (ii) implies (i).  On the other hand, if (i) holds true, then $Y^0:= Y$ is clearly a square integrable semimartingale between $L$ and $U$.    That is, (i) and (ii) are equivalent.

Next, assume  (ii) holds true.  Since $L\le Y^0\le U$, then for any partition $\pi: 0=\t_0<\cds<\t_n=T$,
\beaa
&& L^+ + U^- \le (Y^0)^+ + (Y^0)^- = |Y^0|;\\
&&\Big[\dbE^\dbP_{\t_i}[L_{\t_{i+1}}] - U_{\t_i}\Big]^+ + \Big[L_{\t_i} - \dbE^\dbP_{\t_i}[U_{\t_{i+1}}] \Big]^+ \\
&\le& \Big[ \dbE^\dbP_{\t_i}[Y^0_{\t_{i+1}}] - Y^0_{\t_i}\Big]^+ +\Big[  Y^0_{\t_i} -\dbE^\dbP_{\t_i}[Y^0_{\t_{i+1}}]\Big]^+ =\Big| \dbE^\dbP_{\t_i}[Y^0_{\t_{i+1}}] - Y^0_{\t_i}\Big|.
\eeaa
This implies immediately that $\|(L, U)\|_{\dbP} \le \|Y^0\|_{\dbP}$, and thus (iii) holds.

It remains to prove that (iii) implies (ii).  We first assume \reff{separate} holds. Then it follows from Lemma \ref{lem-HHO} that DRBSDE \reff{2RBSDE} with $f=0$ admits a local solution $(Y^0, Z^0, A^0)$.  Applying Lemma \ref{lem-apriori} we see that $\|(Y^0, Z^0, A^0)\| \le C[I_0 + \|(L,U)\|_\dbP]$. This implies \reff{Mokobodski}. 

In the general case, denote $U^n := U + \frac{1}{n}$. 
Then $(L, U^n)$ satisfies \reff{separate}. By the above arguments,   DRBSDE \reff{2RBSDE} with coefficients $(\xi, 0, L, U^n)$ has a unique solution $(Y^n, Z^n, A^n)$ satisfying
\beaa
\|(Y^n, Z^n, A^n)\|^2 \le C\dbE^\dbP[|\xi|^2] + C \|(L,U^n)\|_\dbP^2
\eeaa
It is obvious that $\|(L, U^n)\|_\dbP\le \|(L,U)\|_\dbP$. Then
\beaa
\|(Y^n, Z^n, A^n)\|^2 \le C\dbE^\dbP[|\xi|^2] + C \|(L, U)\|_{\dbP}^2.
\eeaa
Now for $m>n$, applying Lemma \ref{lem-difference} we have
\beaa
&&\dbE^\dbP\Big[\sup_{0\le t\le T}[|Y^n_t-Y^m_t|^2 + [A^n_t-A^m_t]^2+ \int_0^T |Z^n_t-Z^m|^2dt\Big]\\
&\le& C\Big[\|(Y^n, Z^n, A^n)\|+\|(Y^m, Z^m, A^m)\|\Big][{1\over n} - {1\over m}] \\
&\le& {C\over n}\Big[\Big(\dbE^\dbP[|\xi|^2]\Big)^{1\over 2} +  \|(L,U)\|_\dbP\Big]. 
\eeaa
Send $n\to\infty$, we obtain limit processes $(Y^0, Z^0, A^0)$. Following standard arguments we see that $Y^0$ satisfies the requirement in (ii). 
\qed


\section{Semimartingales under $G$-expectation}
\label{sect-G}
\setcounter{equation}{0}
In this section we introduce a nonlinear expectation, which is a variation of the $G$-expectation proposed by Peng \cite{Peng-G}, and we shall still call it $G$-expectation. Let $(\O, \cF,\dbF)$ be a filtered space such that $\dbF$ is right continuous and $\cP$ be a family of probability measures. For each $\dbP\in\cP$ and $\dbF$-stopping time $\t$, denote
\bea
\label{cPtP}
\cP(\t, \dbP) := \big\{\dbP'\in\cP:  \dbP' = \dbP ~~\mbox{on}~~\cF_\t\big\}.
\eea
Throughout this section, we shall always assume

\begin{assum}
\label{assum-cP}
(i) \reff{Filtration} holds for every $\dbP\in \cP$;

(ii) $\cN_\cP\subset \cF_0$, where $\cN_\cP$ is the set of all $\cP$-polar sets, that is, all $E\in\cF$ such that $\dbP(E)=0$ for all $\dbP\in\cP$.

(iii) For any $\dbP\in\cP$, $\dbF$-stopping time $\t$, $\dbP_1, \dbP_2\in\cP(\t,\dbP)$, and any partition $E_1, E_2 \in \cF_\t$ of $\O$ , the probability measure $\bar\dbP$ defined below also belongs to $\cP(\t,\dbP)$:
\bea
\label{barP}
\bar\dbP(E) := \dbP_1(E\cap E_1) + \dbP_2(E\cap E_2), &\mbox{for all}& E\in \cF.
\eea
\end{assum}
We provide below an important example for such $\cP$, which induces the $G$-expectation of Peng \cite{Peng-G}, and we refer to \cite{STZ-aggregation} for more examples. 
\begin{eg}
\label{eg-G}
Let $\O := \{\o\in C([0, T], \dbR): \o_0=0\}$, $B$ the canonical process, $\dbF$ the right limit of the filtration generated by $B$, $\dbP_0$ the Winer measure. Let $ 0\le \underline\si < \overline\si$ be two constants. For each bounded $\dbF$-progressively measurable process $\si$, denote $X^\a_t := \int_0^t \a_s dB_s$, $\dbP_0$-a.s. Then the following class $\cP$ satisfies Assumption \ref{assum-cP}:
\beaa
\cP := \{\dbP^\si: \underline\si \le \si\le \overline\si\} &\mbox{where}& \dbP^\si := \dbP_0\circ (X^\a)^{-1}.
\eeaa
\end{eg}
\subsection{Definitions}
We first define
\begin{defn}
\label{defn-cPmg}
We say an $\dbF$-progressively measurable process $Y$ is a $\cP$-martingale (resp. $\cP$-supermartingale, $\cP$-submartingale, $\cP$-semimartingale) if it is a $\dbP$-martingale (resp. $\dbP$-supermartingale, $\dbP$-submartingale, $\dbP$-semimartingale) for all $\dbP\in\cP$.
\end{defn}

We next define the $G$-expectation and conditional $G$-expectation.  For any $\cF$-measurable random variable $\xi$ such that $\dbE^\dbP[|\xi|]<\infty$ for all $\dbP\in\cP$,  its $G$-expectation is defined by
\bea
\label{EG}
 \dbE^G[\xi] := \sup_{\dbP\in\cP} \dbE^\dbP[\xi].
 \eea
 The conditional $G$-expectation is more involved.  For any $\dbF$-stopping time $\t$, denote
\bea
\label{EGPt}
\dbE^{G,\dbP}_\t[\xi] := \esup_{\dbP'\in \cP(\t,\dbP)}^{\dbP} \dbE^{\dbP'}_\t[\xi],~~\dbP\mbox{-a.s.} 
\eea
We note that, by Lemma \ref{lem-filtration}, we may take the convention that $\dbE^{G,\dbP}_\t[\xi]$ is $\cF_\t$-measurable. 
When the family $\{\dbE^{G,\dbP}_\t[\xi], \dbP\in\cP\}$ can be aggregated, that is, there exists an $\cF_\t$-measurable random variable, denoted as $\dbE^{G}_\t[\xi]$, such that
\bea
\label{EGt}
\dbE^{G}_\t[\xi] = \dbE^{G,\dbP}_\t[\xi],~~\dbP\mbox{-a.s. for all}~~ \dbP\in\cP,
\eea
we call $\dbE^{G}_\t[\xi]$ the conditional $G$-expectation of $\xi$.  We refer to Soner, Touzi and Zhang  \cite{STZ-aggregation} for the detailed study on the  aggregation issue.  Following standard arguments, we have the following time consistency (or say, Dynamic Programming Principle), whose proof is provided in the Appendix for completeness:

\begin{lem}
\label{lem-DPP}
Under Assumption \ref{assum-cP}, for any $\t_1 \le \t_2$ and any $\dbP\in\cP$, we have
\beaa
\dbE^{G,\dbP}_{\t_1}[\xi] = \esup_{\dbP'\in \cP(\t_1,\dbP)}^{\dbP} \dbE^{\dbP'}_{\t_1}\big[\dbE^{G,\dbP'}_{\t_2}[\xi]\big],~~\dbP-\mbox{a.s.}
\eeaa
\end{lem}

We finally define

\begin{defn}
\label{defn-Gmg}
We say an $\dbF$-progressively measurable process $Y$ is a $G$-martingale (resp. $G$-supermartingale, $G$-submartingale) if, for  any $\dbP\in\cP$ and  any $\dbF$-stopping times $\t_1 \le \t_2$,
  \beaa
  Y_{\t_1} = (\mbox{resp.} \ge, \le) \dbE^{G,\dbP}_{\t_1}[Y_{\t_2}],~~\dbP\mbox{-a.s.}
  \eeaa
\end{defn}
We remark that a $\cP$-martingale is also called a symmetric $G$-martingale in the literature, see e.g. \cite{XZ}.

\subsection{Characterization of $\cP$-semimartingales}
The following result is immediate:

\begin{prop}
\label{prop-cPG}
Let Assumption \ref{assum-cP} hold.

(i) A $\cP$-martingale (resp. $\cP$-supermartingale, $\cP$-submartingale) must be a $G$-martingale (resp. $G$-supermartingale, $G$-submartingale).

(ii) If $Y$ is a $G$-martingale (resp. $G$-supermartingale, $G$-submartingale) and $M$ is a $\cP$-martingale, then $Y+M$ is a $G$-martingale (resp. $G$-supermartingale, $G$-submartingale).

(iii) A $G$-supermartingale is a $\cP$-supermartinagle. In particular, a $G$-martingale is a $\cP$-supermartinagle.
\end{prop}
\begin{proof} (i) and (ii) are obvious. To prove (iii),  let $Y$ be a $G$-supermartingale. Then for any $\t_1\le \t_2$ and any $\dbP\in \cP$,
\beaa
Y_{\t_1} \ge \dbE^{G,\dbP}_{\t_1}[Y_{\t_2}] \ge \dbE^{\dbP}_{\t_1}[Y_{\t_2}],\q \dbP\mbox{-a.s.}
\eeaa
That is, $Y$ is a $\dbP$-supermartingale for all $\dbP\in\cP$, and thus is a $\cP$-supermartingale.
\end{proof}

We next study $\cP$-semimartingales.  In light of Theorem \ref{thm-semimg}, we define a new norm:
\bea
\label{Y-cP}
\|Y\|_\cP := \sup_{\dbP\in\cP} \|Y\|_\dbP.
\eea 
The following result is a direct consequence of Theorems \ref{thm-semimg-est} and \ref{thm-semimg}.

\begin{thm}
\label{thm-semicP}
Let Assumption \ref{assum-cP} hold. If $\|Y\|_\cP<\infty$, then $Y$ is a $\cP$-semimartingale. Moreover, for any $\dbP\in \cP$ and for the decomposition 
\bea
\label{YMAP}
Y_t = Y_0 + M^\dbP_t + A^\dbP_t,~~\dbP\mbox{-a.s.}
\eea
 we have
\beaa
\dbE^\dbP\Big[\la M^\dbP\ra_T + \big(\bigvee_0^T A^\dbP\big)^2\Big] \le C\|Y\|_\cP^2.
\eeaa
\end{thm}

The norm $\|\cd\|_\cP$ is defined through each $\dbP\in\cP$. The following definition relies on the $G$-expectation directly:
\bea
\label{Y-G}
\|Y\|_G^2 :=   \dbE^G\Big[\sup_{0\le t\le T}|Y_t|^2\Big] + \sup_\pi \sup_{\dbP\in\cP}\dbE^\dbP\Big[\Big( \sum_{i=0}^{n-1} \Big |\dbE^{G,\dbP}_{\t_i} (Y_{\t_{i+1}})- Y_{\t_i} \Big| \Big)^2\Big].
\eea
\begin{rem}
\label{rem-normG}
{\rm (i) If the involved conditional $G$-expectations exist, then we may simplify the definition of $\|Y\|_G$:
\beaa
\|Y\|_G^2 :=   \dbE^G\Big[\sup_{0\le t\le T}|Y_t|^2\Big] + \sup_\pi \dbE^G\Big[\Big( \sum_{i=0}^{n-1} \Big |\dbE^{G}_{\t_i} (Y_{\t_{i+1}})- Y_{\t_i} \Big| \Big)^2\Big].
\eeaa

(ii)  In general $\|\cd\|_G$  does not satisfy the triangle inequality and thus is not a norm.

(iii) For $G$-submartingales $Y^1, Y^2$, the triangle inequality holds:
\beaa
\|Y^1+Y^2\|_G \le \|Y^1\|_G + \|Y^2\|_G.
\eeaa
However, in general $Y^1 + Y^2$ may not be a $G$-submartingale anymore.
\qed}
\end{rem}

\no Nevertheless, $\|Y\|_G$ involves the process $Y$ only. The following estimate is the main result of this section.

\begin{thm}
\label{thm-semiG} 
Assume Assumption \ref{assum-cP} holds. Then there exists a universal constant $C$ such that $\|Y\|_\cP \le C\|Y\|_G$.
\end{thm}
\begin{proof} Without loss of generality, we assume $\|Y\|_G<\infty$. For any  $\dbP\in\cP$ and any partition $\pi: 0=\t_0\le \cds\le \t_n=T$, denote
\beaa
N_{\t_i} := \sum_{j=0}^{i-1} \Big[\dbE^{G,\dbP}_{\t_j} (Y_{\t_{j+1}})- Y_{\t_j}\Big].
\eeaa
Then
\beaa
Y_{\t_i} - N_{\t_i} &=& Y_0 +  \sum_{j=0}^{i-1} \Big[ Y_{\t_{j+1}}  - \dbE^{G,\dbP}_{\t_j} (Y_{\t_{j+1}})\Big]\\
& =& Y_0 +  \sum_{j=0}^{i-1} \Big[ Y_{\t_{j+1}}  - \dbE^{\dbP}_{\t_j} (Y_{\t_{j+1}})\Big]- \sum_{j=0}^{i-1} \Big[\dbE^{G,\dbP}_{\t_j} (Y_{\t_{j+1}}) - \dbE^{\dbP}_{\t_j} (Y_{\t_{j+1}})\Big].
\eeaa
Note that 
\beaa
&& \sum_{j=0}^{i-1} \Big[ Y_{\t_{j+1}}  - \dbE^{\dbP}_{\t_j} (Y_{\t_{j+1}})\Big]~~\mbox{is a $\dbP$-martingale},\\
 && \sum_{j=0}^{i-1} \Big[\dbE^{G,\dbP}_{\t_j} (Y_{\t_{j+1}}) - \dbE^{\dbP}_{\t_j} (Y_{\t_{j+1}})\Big]~~\mbox{is nondecreasing and is $\cF_{\t_{i-1}}$-measurable}.
 \eeaa
  Applying Lemma \ref{lem-super2} we obtain
\beaa
\dbE^\dbP\Big[\Big( \sum_{j=0}^{n-1} \big[\dbE^{G,\dbP}_{\t_j} (Y_{\t_{j+1}}) - \dbE^{\dbP}_{\t_j} (Y_{\t_{j+1}})\big]\Big)^2\Big] \le C\dbE^\dbP\Big[\sup_{0\le i\le n}[|Y_{\t_i}|^2 + |N_{\t_i}|^2]\Big] \le C\|Y\|_G^2.
\eeaa
This, together with the definition of $\|\cd\|_G$, implies that
\beaa
\dbE^\dbP\Big[\Big( \sum_{j=0}^{n-1} \big|\dbE^{\dbP}_{\t_j} (Y_{\t_{j+1}})- Y_{\t_j}\big|\Big)^2\Big] \le C\|Y\|_G^2.
\eeaa
Since $\pi$ is arbitrary, we get $\|Y\|_\dbP\le C\|Y\|_G$. Finally, since $\dbP\in\cP$ is arbitrary, we prove the result.
\end{proof}

\subsection{Doob-Meyer Decomposition for $G$-submartingales}

As a special case of Theorem \ref{thm-semicP}, we have the following decomposition for $G$ submartingales.
\begin{prop}
\label{prop-subcP}
Assume Assumption \ref{assum-cP} holds. If $Y$ is a $G$-submartingale satisfying $\|Y\|_\cP<\infty$ (in particular if $\|Y\|_G <\infty$), then all the results in  Theorem \ref{thm-semicP} hold.
\end{prop}

\begin{rem}
\label{rem-DoobMeyer}
{\rm Unlike Lemma \ref{lem-super}, for $G$-submartingales in general we do not have $\|Y\|_\cP \le C\sup_{\dbP\in\cP}\|Y\|_{\dbP,0}$. See Example \ref{eg-Gsub} below.
\qed}
\end{rem}

Now let $Y$ be as in Proposition \ref{prop-subcP}, and consider its decomposition \reff{YMAP}. Let $A^\dbP = L^\dbP - K^\dbP$ be the orthogonal decomposition. We have the following conjecture:

\bs
\no{\bf Conjecture (Doob-Meyer decomposition)} : {\it The family $\{K^\dbP, \dbP\in\cP\}$ satisfies the following property:
\bea
\label{Kmg}
- K^\dbP_t = \esup_{\dbP'\in \cP(t,\dbP)}^{\dbP} \dbE^{\dbP'}_t\Big[- K^{\dbP'}_T\Big].
\eea
In particular, if the families $\{M^\dbP, K^\dbP, L^\dbP, \dbP\in \cP\}$ can be aggregated into $\{M, K, L\}$,  then $-K$ is a $G$-martingale, and we have the following desired Doob-Meyer decomposition for $G$-submartingales:
\bea
\label{DM}
\left.\ba{c}
Y_t = Y_0 + [M_t - K_t] + L_t, \\
\mbox{where $M-K$ is a $G$-martingale and $L$ is nondecreasing}.
\ea\right.
\eea
\qed}

We refer again to \cite{STZ-aggregation} for the issue of aggregation. In particular, we can always aggregate the families $\{M^\dbP, K^\dbP, L^\dbP, \dbP\in \cP\}$ when the class $\cP$ is separable, in the sense of \cite{STZ-aggregation}.  This conjecture looks natural, but it is quite subtle. Our estimates in this section are rather preliminary. We hope to address the issue more thoroughly in some future research.


\section{Appendix}
\setcounter{equation}{0}

We first provide an example such that $\|Y\|_{\dbP,0} <\infty$ but $\|Y\|_\dbP =\infty$.
\begin{eg}
\label{eg-Y0}
Fix $\dbP$. Let $K$ be an $\dbF$-progressively measurable continuous increasing  process such that $K_0 = 0$ and $\dbE^\dbP[K^2_T] = \infty$. Define the sequence of stopping times: $\t_0:=0$ and, for $n\ge 1$, $\t_n := \inf \{ t \ge 0: K_t = n \} \wedge T$. Since $K_T < \infty$, $\t_n = T$ for $n$ large enough, a.s. We now define the process $Y_t$ as follows: $Y_0:=0$, and for $n\ge 0$,
\bea
\label{eg1-Y}
Y_t := \left \{ 
\ba{lll}
Y_{\t_{2n}} - K_t + K_{\t_{2n}},\q t\in  (\t_{2n}, \t_{2n+1}] ;\\
  Y_{\t_{2n+1}} + K_t-K_{\t_{2n+1}} ,\q t \in (\t_{2n+1}, \t_{2n+2}] .
  \ea\right.
  \eea
  Then $\|Y\|_{\dbP,0} <\infty$ but $\|Y\|_\dbP =\infty$.
  \end{eg}
  \begin{proof}
It is easy to check  that $-1 \le Y_t \le 0$ and  $\bigvee_0^T Y = K_T$. Then $\|Y\|_{\dbP,0} \le 1$ and $\dbE^\dbP \Big[\big( \bigvee_0^T Y  \big)^2\Big] = \infty$. By Theorem \ref{thm-semimg}, we get $\|Y\|_\dbP=\infty$.
\end{proof}

We next provide a $G$-submartingale such that $\sup_{\dbP\in \cP}\|Y\|_{\dbP,0} <\infty$, but $\|Y\|_{\cP}=\infty$.
\begin{eg}
\label{eg-Gsub}
Fix $\cP$. Let $K$ be as in Example \ref{eg-Y0} such that $-K$ is a $G$ martingale and $\dbE^G[K^2_T] = \infty$, instead of $\dbE^\dbP[K^2_T] = \infty$. Then the process $Y$ defined in Example \ref{eg-Y0} satisfies all the requirements. 
\end{eg}
\begin{proof} By the proof of Example \ref{eg-Y0}, clearly $\sup_{\dbP\in \cP}\|Y\|_{\dbP,0} <\infty$, but $\|Y\|_{\cP}=\infty$. Moreover, on $(\t_{2n}, \t_{2n+1}]$, $dY_t = -dK_t$ and thus is a $G$martingale; and on  $(\t_{2n+1}, \t_{2n+2}]$, $dY_t = dK_t$, then $Y$ is increasing and thus is a $G$-submartingale. So $Y$ is a $G$-submartingale on $[0,T]$. 
\end{proof}

We finally prove Lemma \ref{lem-DPP}.

\ms
{\it Proof of Lemma \ref{lem-DPP}.}
First we have
\beaa
\dbE^{G,\dbP}_{\t_1}[\xi] &=& \esup_{\dbP'\in \cP(\t_1,\dbP)}^{\dbP} \dbE^{\dbP'}_{\t_1}[\xi] = \esup_{\dbP'\in \cP(\t_1,\dbP)}^{\dbP} \dbE^{\dbP'}_{\t_1}[\dbE^{\dbP'}_{\t_2} [\xi]] \leq \esup_{\dbP'\in \cP(\t_1,\dbP)}^{\dbP} \dbE^{\dbP'}_{\t_1}[\dbE^{G,\dbP'}_{\t_2} [\xi]].
\eeaa

To prove the other inequality,  for each $\dbP'\in \cP(\t_1, \dbP)$, we recall from Neveu \cite{Neveu} that there exists a sequence $\dbP^n \in \cP(\t_2,\dbP')$ such that
\beaa
\dbE^{\dbP^n}_{\t_2} [\xi] \rightarrow \esup_{\tilde{\dbP} \in \cP(\t_2,\dbP')}^{\dbP'} \dbE^{\tilde{\dbP}}_{\t_2}[\xi],~~\mbox{as}~~ n \rightarrow \infty.
\eeaa
We claim that we may construct  $\tilde{\dbP}^n \in \cP(\t_2,\dbP')\subset \cP(\t_1, \dbP)$ such that 
\bea
\label{tildePn}
\dbE^{\tilde\dbP^n}_{\t_2} [\xi]  = \max_{1\le m\le n} \dbE^{\dbP^m}_{\t_2} [\xi], ~\dbP'\mbox{-a.s.}
\eea
Then clearly  $\dbE^{\tilde{\dbP}^n}_{\t_2} [\xi] \uparrow \esup_{\tilde{\dbP} \in \cP(\t_2,\dbP')}^{\dbP'} \dbE^{\tilde{\dbP}}_{\t_2}[\xi]$, $\dbP'$-a.s.. Thus, 
\beaa
\dbE^{\dbP'}_{\t_1}\Big[ \esup_{\tilde{\dbP} \in \cP(\t_2,\dbP')}^{\dbP'} \dbE^{\tilde{\dbP}}_{\t_2}[\xi] \Big] &=& \lim_{n \rightarrow \infty} \dbE^{\dbP'}_{\t_1} \Big[ \dbE^{\tilde{\dbP}^n}_{\t_2} [\xi] \Big] \\
&=& \lim_{n \rightarrow \infty} \dbE^{\tilde{\dbP}^n}_{\t_1} \Big[\dbE^{\tilde{\dbP}^n}_{\t_2} [\xi] \Big] = \lim_{n \rightarrow \infty}\dbE^{\tilde{\dbP}^n}_{\t_1}[\xi] \leq \dbE^{G,\dbP}_{\t_1}[\xi],~~\dbP\mbox{-a.s.}
\eeaa
By the arbitrariness of  $\dbP'\in \cP(\t_1, \dbP)$, we prove the lemma.

It remains to prove \reff{tildePn}. Indeed, for $\dbP^1, \dbP^2 \in \cP(\t_2,\dbP')$,  define
\beaa
\tilde{\dbP}^2 (E) := \dbP^1 (E \cap E^+) + \dbP^2 (E \cap E^-), &\mbox{for any}& E\in \cF,
\eeaa
where
$E^+ := \Big \{ \dbE^{\dbP^1}_{\t_2} [\xi] \geq \dbE^{\dbP^2}_{\t_2} [\xi] \Big \}$ and $E^- := \Big \{ \dbE^{\dbP^1}_{\t_2} [\xi] < \dbE^{\dbP^2}_{\t_2} [\xi] \Big \}$. Clearly $E^+, E^- \in \cF_{\t_2}$ and form a partition of $\O$. Then it follows from Assumption \ref{assum-cP} (iii) that $\tilde\dbP^2\in \cP(\t_2, \dbP')$.   Moreover, for any $E \in \cF_{\t_2}$:
\beaa
\dbE^{\dbP'}\Big[\dbE^{\tilde{\dbP}^2}_{\t_2} [\xi]\mathbf{1}_E\Big] &=&\dbE^{\tilde{\dbP}^2} [\xi\mathbf{1}_E] = \dbE^{\tilde{\dbP}^2} [\xi\mathbf{1}_{E \cap E^+}] + \dbE^{\tilde{\dbP}^2} [\xi\mathbf{1}_{E \cap E^-}]\\
&=& \dbE^{\dbP^1} [\xi\mathbf{1}_{E \cap E^+}] + \dbE^{\dbP^2} [\xi\mathbf{1}_{E \cap E^-}]\\
&=& \dbE^{\dbP'}\Big[\dbE^{\dbP^1}_{\t_2} [\xi]\mathbf{1}_{E \cap E^+}\Big] +  \dbE^{\dbP'}\Big[\dbE^{\dbP^2}_{\t_2} [\xi]\mathbf{1}_{E \cap E^-}]\Big]\\
& =&   \dbE^{\dbP'}\Big[\big( \dbE^{\dbP^1}_{\t_2} [\xi] \vee \dbE^{\dbP^2}_{\t_2} [\xi]\big) \1_E\Big] .
\eeaa
Thus $\dbE^{\tilde{\dbP}^2}_{\t_2} [\xi] = \dbE^{\dbP^1}_{\t_2} [\xi] \vee \dbE^{\dbP^2}_{\t_2} [\xi] $. Repeat the arguments we prove \reff{tildePn} and hence the lemma.
\qed



\begin{thebibliography}{1}


\bibitem{Chassagneux}
Chassagneux, J.  (2009) {\it A discrete-time approximation for doubly reflected BSDE}, {\sl Advances in Applied Probability}, 41,  101-130.


\bibitem{CSTV}
Cheridito, P., Soner, H.M. and Touzi, N., Victoir, N. (2007) {\it Second order BSDE's and fully nonlinear PDE's}, {\sl Communications
in Pure and Applied Mathematics}, 60 (7): 1081-1110.

\bibitem{CK}
Cvitanic, J. and Karatzas, I. (1996) {\it Backward SDEÕs with reflection and Dynkin games}, {\sl Annals of Probability},  24, 2024-2056.

\bibitem{CZ}
Cvitanic, J. and Zhang, J.  (2012) {\sl Contract Theory in Continuous Time Models}, Springer Finance.




\bibitem{DM}
Dellacherie, C. and  Meyer, P. A. (1982), {\sl Probabilities and Potential B}, North-Holland, Amsterdam.

\bibitem{EKTZ}
Ekren, I., Keller, C., Touzi, N., and Zhang, J.  {\it On Viscosity Solutions of Path Dependent PDEs},  {\sl Annals of Probability}, to appear, arXiv:1109.5971.

\bibitem{ETZ0}
Ekren, I., Touzi, N., and Zhang, J.  {\it Optimal Stopping under Nonlinear Expectation}, preprint, arXiv:1209.6601.

\bibitem{ETZ1}
Ekren, I., Touzi, N., and Zhang, J.  {\it Viscosity Solutions of Fully Nonlinear Parabolic Path Dependent PDEs: Part I}, preprint,  arXiv:1210.0006.

\bibitem{ETZ2}
Ekren, I., Touzi, N., and Zhang, J.  {\it Viscosity Solutions of Fully Nonlinear Path Parabolic Dependent PDEs: Part II}, preprint, arXiv:1210.0007.


\bibitem {EKPPQ}
El. Karoui, N., Kapoudjian, C., Pardoux, E., Peng, S. and Quenez,
M. (1997) {\it Reflected Solutions of Backward SDE's, and Related
Obstacle Problems for PDE's}, {\sl The Annals of Probability}, 
25, 702-737.






\bibitem{HH}
Hamad\'{e}ne, S. and Hassani, M. (2005) {\it BSDEs with two reflecting barriers: the general result}, {\sl Probability Theory and Related Fields}, 132, 237-264.

\bibitem{HHO}
Hamad\'{e}ne, S., Hassani, M., and Ouknine, Y. (2010) {\it BSDEs with general discontinuous reflecting barriers without Mokobodski's condition}, {\sl Bull. Sci. math.}, 134, 874-899 ; DOI : 10.1016/j.bulsci.2010.03.001

\bibitem{HJPS}
Hu, M., Ji, S., Peng, S., and Song, Y. (2012) {\it Backward Stochastic Differential Equations Driven by G-Brownian Motion}, preprint, arXiv:1206.5889.



\bibitem{KS}
Karatzas,I and Shreve, S. {\it Brownian Motion and Stochastic Calculus}, 2nd Edition, Springer.

\bibitem{MZ}
Meyer, P. and Zheng, W. (1984) {\it Tightness criteria for laws of
semimartingales}, {\sl Ann. Inst. Henri Poincar\'{e}}, 20, 353-372.

\bibitem{Neveu}
Neveu, J. (1975) {\it Discrete Parameter Martingales.} North Holland Publishing Company.




\bibitem{Peng-g}
Peng, S. (1997) {\it Backward SDE and related gÐexpectation}, {\sl Backward stochastic differential equations}, (N. El Karoui and L. Mazliak, eds.), Pitman Res. Notes Math. Ser., vol. 364, Long- man, Harlow, p. 141Ð159.





\bibitem{Peng-G}
Peng, S. (2010) {\it Nonlinear Expectations and Stochastic Calculus under Uncertainty}, preprint, arXiv:1002.4546.



\bibitem{PX}
Peng, S. and Xu, M. (2005) {\it The smallest $g$-supermartingale and reflected BSDE with single and double $L^2$ obstacles}, {\sl Annales de I.H.P.},  141, 605-630.

\bibitem{RY} Revuz, D. and Yor, M.  (1999) {\it Continuous martingales and Brownian motion,} {\sl Springer}, third edition.

\bibitem{PZ}
Pham, T. and Zhang, J. (2012) {\it  Two Person Zero-sum Game in Weak Formulation and Path Dependent Bellman-Isaacs Equation},  preprint, arXiv:1209.6605.

\bibitem{Protter}
Protter, P. (2004) {\sl Stochastic Integration and Differential Equations}, 2nd Edition, Springer.

\bibitem{Rao}
Rao, K. M. (1969),  {\it Quasi-Martingales}, {\sl Math. Scand.},  24, 79-92.

\bibitem{STZ-G}
Soner, M., Touzi,  N. and Zhang, J. (2011)  {\it Martingale representation theorem for the $G-$expectation}, {\sl Stochastic Processes and Their Applications}, 121, 265-287.

\bibitem{STZ-aggregation}
Soner, M., Touzi,  N. and Zhang, J. (2011)   {\it  Quasi-sure stochastic analysis through aggregation}, {\sl Electronic Journal of Probability}, 16, 1844-1879.


\bibitem{STZ-2BSDE}
Soner, M., Touzi,  N. and Zhang, J. (2012), {\it Wellposedness of Second Order BSDEs}, {\sl Probability Theory and Related Fields}, 153, 149-190.



\bibitem{XZ}
Xu, J. and Zhang, B. (2009) {\it Martingale characterization of G-Brownian motion}, {\sl Stochastic Processes and their Applications}, 119, 232-248.

\end{thebibliography}
\end{document}